\documentclass[12pt]{article}

\usepackage{latexsym}
\usepackage{amssymb}
\usepackage{euscript}
\usepackage{amsmath}
\usepackage[noblocks]{authblk}
\usepackage{xfrac}
\usepackage{enumerate}
\usepackage{index}
\usepackage{xcolor}
\usepackage{graphics}
\usepackage{pictex}
\usepackage[all]{xy}
\usepackage{multicol} 
\usepackage{amsthm}

\newcommand {\N} {{\rm I}\hspace{-1mm}{\rm N}}
\newcommand{\gen}{\reflectbox{$\neg$}}
\newcommand{\set}[1]{\{#1\}}
\newcommand{\lat}{{\bf L}}
\newcommand{\pos}{{\bf P}}
\newcommand{\fx}{{\sf X}}
\newcommand{\fd}{{\sf D}}
\newcommand{\mnu}{^{-1}}
\newcommand{\Ns}{{\bf Ns}}

\newcommand{\uc}{\tfrac{1}{4}}
\newcommand{\cc}{\tfrac{1}{2}}
\newcommand{\tc}{\tfrac{3}{4}}

\renewcommand{\lceil}{\gen}

\newtheorem{definition}{\bf Definition}[section]
\newtheorem{example}[definition]{\bf Example}
\newtheorem{lemma}[definition]{\bf Lemma}
\newtheorem{theorem}[definition]{\bf Theorem}
\newtheorem{remark}[definition]{\bf Observation}
\newtheorem{corollary}[definition]{\bf Corollary}

\title{Free five-valued Nelson Algebras}
\author[1,2]{  Juan Manuel Cornejo}
 \author[1,2]{Andr\'es Gallardo}
  \author[2]{Luiz F. Monteiro}  
  \author[1,2]{Ignacio Viglizzo}
  \affil[1]{INMABB-CONICET-Universidad Nacional del Sur}
  \affil[2]{Departamento de Matem\'atica-Universidad Nacional del Sur}

\begin{document}
	\maketitle

	\hfill {\it In memoriam of Diana Brignole and Antonio Monteiro}

	\thispagestyle{empty}
	\begin{abstract}
		Five-valued Nelson algebras are those satisfying the condition: $((x\to z)\to y)\to(((y \to x)\to y)\to y)=1$.  
		We give alternative equations defining these algebras, and determine the structure and number of elements of the free five-valued Nelson algebra with a finite number of free generators.  
	\end{abstract}

\section{Introduction}\label{S1}
	
  In this paper we present a complete description of  the free five-valued Nelson Algebras with a finite number of generators (see below for the definition). The  problem of determining them was posed by  Antonio Monteiro  to Diana Brignole some time before 1964. At that moment, A.~Monteiro was the director of the INMABB, the Mathematics Institute of Bah\'ia Blanca, and Brignole was one of his students. In 1965, D. Brignole presented a communication in the annual meeting of the UMA (Uni\'on Matem\'atica Argentina) \cite{brignole65algebras}, where she indicated some initial results on these algebras, and a preprint submitted to CONICET \cite{brignole65nelson} (a copy of which we recently found among the correspondence of A. Monteiro), which includes other results and a formula for the number of elements in the free five-valued Nelson algebra $F(n)$, with a set of $n$ free generators, but without a proof. At the end of this preprint, Brignole states, ``The proofs of these results will be indicated in a next paper'' but this next paper was never completed. 
	
	Later on, Brignole continued her studies in the USA, with a grant from CONICET (Consejo Nacional de Investigaciones Científicas y Técnicas, Argentina). During this period, on September 16, 1964,  she wrote a letter \footnote{A copy is held at the National Library of Portugal, Arquivo da Cultura Portuguesa Contemporânea, Campo Grande, Lisbon.} to Dr. A. Monteiro saying {\it ``como Ud. seguramente se imagina yo no termin\'e de demostrarla''}: ``as you surely imagine, I did not finish proving it''. 
	
	In this work we present the previously unpublished results of D.~Brignole in greater detail, as well as new results that are necessary to determine $F(n)$. We prove that the formula provided by D.~Brignole is correct and determines the number of elements in  $F(n)$. In the Appendix we reproduce verbatim the two works by D. Brignole \cite{{brignole65algebras},{brignole65nelson}}, related to this topic.  

It should be emphasized that even tough in this work we use the concept of Priestley representation to give a more modern presentation, it was not necessary, and the determination of the free algebras could have been done by Brignole using the representation of a lattice by its prime elements and filters, with the Birula-Rasiowa operation and the interpolation property. 

\section{Preliminaries}\label{S2}

The notion of Nelson algebra or N-lattice was introduced by H. Rasiowa \cite{rasiowa74algebraic}, and plays a role in the study of constructive propositional calculus with strong negation that is analogous to the role of Boolean algebras in classical propositional calculus. It is due to H. Rasiowa's decisive contribution that constructive propositional calculus with strong negation can be considered a special chapter of algebra, which specifically aims to study Nelson algebras. In 1962, A. Monteiro taught a course at the Universidad Nacional del Sur \cite{monteiro62algebra}, where he presented new results, which, according to his words, were ``obtained in the second semester of 1961, during my stay as a visiting professor at the University of Buenos Aires''.

In the following years, some of these results were published (\cite{monteiro62algebra,monteiro63nelson,monteiro64constructionnotas,monteiro78reguliers,monteiro78lineaires,monteiro80reguliers,monteiro95reguliers,monteiro95heyting,monteiro95nelson,monteiro96nelson}), while others remained unpublished. During those years,  Diana Brignole was his main collaborator, and together they obtained an equational definition of Nelson algebras, \cite{brignole67caracterisation}, using Zorn's Lemma. Later, D. Brignole, \cite{brignole69equational}, provided a purely arithmetic proof of this result. A.~Monteiro and L.~Monteiro, \cite{monteiro96axiomes}, established independent axioms for these algebras, which are given below.

\begin{definition}
	An algebra $(N,1,\sim ,\land ,\lor ,\to )$ with arities $(0,1,2,2,2)$ is said to be an N-lattice or a Nelson algebra if the following equations are satisfied:
	
	\begin{multicols}{2}
	\begin{enumerate}[\normalfont(N1)]
		\item $x\land(x\lor y)=x$,\label{nelsonax1}
		\item $x\land(y\lor z)=(z\land x)\lor(y\land x)$,\label{nelsonax2}
		\item $\sim\sim x=x$,\label{nelsonax3}
		\item $\sim(x\lor y)=\sim x\lor\sim y$,\label{nelsonax4}
		\item $x\land\sim x=(x\land\sim x)\land(y\lor\sim y)$,\label{nelsonax5}
		\item $x\to x=1$,\label{nelsonax6}
		\item $x\to(y\to z)=(x\land y)\to z$,\label{nelsonax7}
		\item $x\land(x\to y)=x\land(\sim x\lor y)$.\label{nelsonax8}
	\end{enumerate}
	\end{multicols}

 For brevity, we say that  $N$ is a   Nelson algebra.
\end{definition}

We can write $0=\sim 1$, and it follows from (N1) to (N4) that $(N,0,1,\sim ,\land ,\lor )$ is a De Morgan algebra, and by (N5), $N$ is a Kleene algebra as well (see \cite{monteiro96axiomes}).

The following quasi-identities hold in a Nelson algebra and will be used as calculation rules throughout the paper. The proofs can be found in A.~Monteiro, \cite{{monteiro63nelson},{monteiro78reguliers},{monteiro78lineaires},{monteiro80reguliers},{monteiro95reguliers},{AM80},{monteiro95nelson},{monteiro96nelson}}, and I.~Viglizzo, \cite{viglizzo99algebras}. The operation $\gen$ is defined by $\gen x=x\to 0$ for all $x$ in a Nelson algebra.

\begin{enumerate}[\normalfont(N1)]
	\setcounter{enumi}{8}
	\item $(x\lor y)\to z= (x\to z)\land (y\to z)$, (A. Monteiro, \cite{monteiro74construction}),\label{111220242}
	\item $x \to(y \to z)= y \to (x \to z)$, \label{detesto}
	\item $x \to(y \to z)= (x\to y) \to (x \to z)$,\label{111220243}
	\item $ \lceil x \leq x\to y$, \label{N19}
	\item $(x\land\sim x)\to y=1$,\label{111220244}
	\item $\sim x , y \leq x\to y$,\label{111220245} 
	\item if $x \leq y$, then $y\to z \leq x \to z$,\label{111220246} 
	\item $1\to x = x$,\label{111220247} 
	\item if $x\leq y$ then $x \to y=1$,\label{111220248} 
	\item $x = y$ if and only if $x\to y=y\to x=\sim y\to \sim x=\sim x\to \sim y=1$, \label{1112202411} 
	\item if $X,Y \subseteq N$ then $\sim (X\cap Y)= \sim X\cap \sim Y$, where $\sim X=\set{\sim x: x\in X}$,\label{111220249} 
	\item $x\to(x\to y)=x\to y$.\label{1112202410}
\end{enumerate}

\subsection{Five-valued Nelson algebras}

Ivo Thomas made in \cite{thomas62finite} a significant contribution to the study of Dummett's $LC$ logic by providing finite axiomatizations for the G\"odel logics that correspond to finite-valued interpretations of $LC$.  He gave formulas that allow to transition from the infinite-valued semantics of Dummett's $LC$ to the finite-valued semantics of Gödel logics. His  axioms constrain the length of implicational chains as follows: let $p_0, p_1, ..., p_{n-1}$ be  well-formed formulas of this calculus and let us define
\[\beta_i = (p_i \rightarrow p_{i+1}) \rightarrow p_0, \mbox{ for } i = 0, 1, ..., n-2\]
then the axiom of Ivo Thomas is:
\[\beta_{n-2} \rightarrow (\beta_{n-3} \rightarrow (... \rightarrow (\beta_0 \rightarrow p_0) ...))\]

Adapting this idea to the particular case of Nelson algebras, with $n=3$ we get the equation:
$$  ((x\rightarrow z)\rightarrow y)\rightarrow(((y \rightarrow x)\rightarrow y)\rightarrow y)=1.$$

We will prove in the next section that the variety of Nelson algebras that satisfy this equation is generated by a chain with five elements, which led Antonio Monteiro to call these algebras five-valued Nelson algebras.

\begin{definition}
	\label{D21}
	A Nelson algebra $N$ is said to be five-valued if for every $x,y,z\in N$ the following equation holds: \[{\rm (NT_3)}\hspace{5mm} ((x\to z)\to y)\to(((y \to x)\to y)\to y)=1.\]
\end{definition}

This equation is analogous to the  one that defines three-valued Heyting algebras. \cite{monteiro72algebre}

\begin{example}
	For any $n \in \N$, $n\geq 2$, let $C_n=\{\displaystyle{\tfrac{j}{n-1}}: 0\leq j \leq n-1\}$. This set with its usual ordering is a bounded distributive lattice, with $0$ as its bottom element and   $1$ as the top.  If we put  $\sim \displaystyle{\tfrac{j}{n-1}}= 1- \displaystyle{\tfrac{j}{n-1}}=\displaystyle{\tfrac{n-1-j}{n-1}}$ for $ 0\leq j \leq n-1$, then 
	$C_n$ is a  Kleene algebra. If $n$ is odd and $j= \displaystyle{\tfrac{n-1}{2}}$ then $\sim \displaystyle{\tfrac{j}{n-1}}= \displaystyle{\tfrac{j}{n-1}}$, that is,  $C_n$ is a centered Kleene algebra.
	
	If   $x,y \in C_n$, we define 
	$$x \to y =
	\left\{
	\begin{array}{ll}
		1  & \mbox{\rm  if } x\leq y  \mbox{\rm\ or\ } x\le\sim x, \\
		\sim x \lor y & \mbox{\rm otherwise.} 
	\end{array}
	\right.
	$$
	
	Thus endowed, $C_n$ is a  Nelson algebra. In particular, the operations $\sim$ and $\to$ for  $C_5=\{0,\uc,\cc,\tc,1\}$ are indicated in this table: 
	$$
	\begin{tabular}{c|ccccc|c}
		$\to$ &  $0$  & $\uc$ & $\cc$ & $\tc$ & $1$ & $\sim x$ \\
		 \hline
		 $0$  &  $1$  &  $1$  &  $1$  &  $1$  & $1$ &   $1$       \\[2mm]
		$\uc$ &  $1$  &  $1$  &  $1$  &  $1$  & $1$ &  $\tc$       \\[2mm]
		$\cc$ &  $1$  &  $1$  &  $1$  &  $1$  & $1$ &  $\cc$       \\[2mm]
		$\tc$ & $\uc$ & $\uc$ & $\cc$ &  $1$  & $1$ &  $\uc$        \\[2mm]
		 $1$  &  $0$  & $\uc$ & $\cc$ & $\tc$ & $1$ &   $0$    
	\end{tabular}
	$$
	It is easy to check that {\rm(NT$_3$)} holds in $C_5$.
\end{example}

\begin{remark}
	\label{R32}
	The Nelson subalgebras of  $C_5$, are $S_5=C_5$, $S_2=\{0,1\}\cong C_2$, \linebreak[4] $S_3=\{0,\tfrac{1}{2},1\}\cong C_3$, and $S_4=\{0,\tfrac{1}{4},\tfrac{3}{4},1\}\cong C_4$.  Therefore $ {\rm NT}_3$ holds in all of them. 
\end{remark}

\begin{example}
	Equation ${\rm (NT}_3)$ does not hold on $C_6$. Indeed, for $x=\tfrac{3}{5}, y=\tfrac{4}{5}$, and $z=\tfrac{2}{5}$, we have that  $((x\to z)\to y)\to(((y \to x)\to y)\to y)=\tfrac{4}{5}\neq 1$.
\end{example}

\begin{remark}
	\label{R51} If $N$ is a Nelson algebra such that $|N|=t$, where $2\leq t\leq 5$, we can analyze the structure of its underlying distributive lattice.  For $t=5$ the lattice must be isomorphic to either  $\{0\}\oplus (C_2 \times C_2)$,  $ (C_2 \times C_2)\oplus \{1\}$ or $C_5$, but in the first two cases, it is not possible to define the involution $\sim$. Thus we can conclude that  $N\cong C_t$ for $2\leq t\leq 5$   or $N\cong C_2\times C_2$ .
\end{remark}

In a similar fashion, three-valued Heyting algebras are defined as those satisfying
$$ {\rm (T}_3) \;\;\;((x\Rightarrow z)\Rightarrow y)\Rightarrow(((y \Rightarrow x)\Rightarrow y)\Rightarrow y)=1,$$
where $\Rightarrow$ is the implication operation in Heyting algebras. 

L. Monteiro proved in \cite{monteiro72algebre} that this condition is equivalent to any of the following conditions:
\begin{enumerate}[\normalfont(C1)]
	\item $(\neg x\Rightarrow y)\Rightarrow(((y \Rightarrow x)\Rightarrow y)\Rightarrow y)=1$,
	\item $((y \Rightarrow x)\Rightarrow y)\Rightarrow ((\neg x\Rightarrow y)\Rightarrow y)=1$,
	\item $y=(\neg x\Rightarrow y)\wedge ((y \Rightarrow x)\Rightarrow y)=(\neg x \vee (y\Rightarrow x))\Rightarrow y$,
	\item $y= ((x\Rightarrow z)\Rightarrow y)\wedge((y \Rightarrow x)\Rightarrow y)$.
\end{enumerate}

More recently,  A. Petrovich and C. Scirica, \cite{PS22} proved that three-valued Heyting algebras can be characterized as the ones which verify the condition:
\[{\rm (PS)}\;\;(x \vee \neg x) \vee (y \vee \neg y) \vee[(x\Rightarrow  y) \wedge (y\Rightarrow  x)]=1.\]
Luiz Monteiro proved in an unpublished note  that the condition  {\rm (PS)} and \[{\rm (LMN)}\;\; x \vee \neg y \vee (x\Rightarrow y)=1\] are equivalent.

We will present now the analogous proofs for the equivalence of all these conditions for Nelson algebras.

\begin{theorem}\label{equivalenciasFN}
	In a Nelson algebra $N$, the following conditions are equivalent for every $x, y, z \in N$:
	\begin{description}
		\item[(NT$_3$)] $((x\rightarrow z)\rightarrow y)\rightarrow(((y \rightarrow x)\rightarrow y)\rightarrow y)=1$,
		\item[(FN1)] $x \vee \neg y \vee (x\rightarrow y)=1$,
		\item[(FN2)]  $y= ((x\rightarrow z)\rightarrow y)\wedge((y \rightarrow x)\rightarrow y)$,
		\item[(FN3)] $y=(\lceil x\rightarrow y)\wedge ((y \rightarrow x)\rightarrow y)$,
		\item[(FN4)] $((y \rightarrow x)\rightarrow y)\rightarrow ((\lceil x\rightarrow y)\rightarrow y)=1$,
		\item[(FN5)] $(\lceil x\rightarrow y)\rightarrow(((y \rightarrow x)\rightarrow y)\rightarrow y)=1$.
	\end{description}
\end{theorem}
\begin{proof}
	\textbf{(NT$_3$)} implies \textbf{(FN1)}:
	
	We prove that $1$ is the least upper bound of the set $\{y, \lceil x, y \to x \}$.
	
	Let  $z$ be such that  (1) $\lceil x \leq z$, (2) $ y \leq z$,  and (3) $y \to x \leq z$.\\
	From  (1), by (N\ref{111220248}) we get (4) $\lceil x \to z=1$. \\
	From  (2), by (N\ref{111220246}) we get $z\to x \leq y \to x{\stackrel{{\rm (3)}}{\leq}}z $, so (5) $(z\to x)\to  z=1$. Then 	
	$$1 {\stackrel{{\rm ({\rm NT}_3)}}{=}} ((x\to 0)\to z)\to (((z\to x)\to z)\to z)=$$
	$$(\lceil x\to z)\to ((z\to x)\to z)\to z){\stackrel{{\rm (4)}}{=}}1\to ((z\to x)\to z)\to z)=$$
	$$((z\to x)\to z)\to z{\stackrel{{\rm (5)}}{=}}1\to z=z,$$
	where the last two lines are obtained applying (N\ref{111220247}).
	
	\textbf{(FN1)} implies \textbf{(FN2)} 
	
	$$y\vee \lceil  x \vee (y\to x){\stackrel{{\rm (FN1)}}{=}}1$$
	so
	$$(y\vee \lceil  x \vee (y\to x))\to y=1\to y{\stackrel{{\rm(N\ref{111220247})}}{=}}y$$
	hence by (N\ref{111220242})
	$$(y\to y)\wedge (\lceil  x\to y)  \wedge  ((y\to x)\to y)=y.$$
	Then by (N\ref{nelsonax6}) 
	$$ (\lceil  x\to y)  \wedge  ((y\to x)\to y)=y.$$
	Therefore 
	$$ ((x \to z)\to y)\wedge (\lceil  x\to y)  \wedge  ((y\to x)\to y)=((x \to z)\to y)\wedge y{\stackrel{{\rm(N\ref{111220245})}}{=}}y.$$
	By (N\ref{111220242}) 
	$$ ((x \to z)\vee \lceil  x)\to y  )\wedge  ((y\to x)\to y)=y,$$
	so by (N\ref{N19}) we obtain
	$$ ((x \to z)\to y  )\wedge  ((y\to x)\to y)=y.$$
	
	\textbf{(FN2)} implies \textbf{(FN3)}. 
	
	Putting $z=0$ in (FN2) we get
	$$y= ((x\rightarrow 0)\rightarrow y)\wedge((y \rightarrow x)\rightarrow y)= ((\gen x\rightarrow y)\wedge((y \rightarrow x)\rightarrow y).$$
	
	\textbf{(FN3)} implies \textbf{(FN4)}. 
	$$y{\stackrel{{\rm (FN3)}}{=}} ((y \rightarrow x)\rightarrow y) \wedge (\gen x\rightarrow y)$$
	$$1{\stackrel{({\rm N\ref{nelsonax6})}}{=}}y \rightarrow y= [((y \rightarrow x)\rightarrow y) \wedge (\gen x\rightarrow y)]\rightarrow y{\stackrel{{\rm(N\ref{nelsonax7})}}{=}}
	((y \rightarrow x)\rightarrow y)\rightarrow ((\gen x\rightarrow y)\rightarrow y).$$
	
	\textbf{(FN4)} is equivalent to  \textbf{(FN5)}. 
	
	By (N\ref{detesto}),
	$$((y \rightarrow x)\rightarrow y)\rightarrow ((\gen x\rightarrow y)\rightarrow y)=(\gen x\rightarrow y)\rightarrow (((y \rightarrow x)\rightarrow y)\rightarrow y),$$ so the equivalence follows.
	
	\textbf{(FN5)} implies \textbf{(NT$_3$)}.
	
	$\gen x{{\rm\stackrel{(N\ref{N19})}{\leq}}}x\rightarrow z$ so by (N\ref{111220246}) $(x\rightarrow z) \rightarrow y \leq \neg x \rightarrow y,$ and using (N\ref{111220246}) again we get:
	
	$$1{{\rm\stackrel{(FN5)}{=}}}(\gen x \rightarrow y)\rightarrow (((y \rightarrow x)\rightarrow y)\rightarrow y)\leq ((x \rightarrow z)\rightarrow y)\rightarrow (((y \rightarrow x)\rightarrow y)\rightarrow y).$$
\end{proof}

\section{Prime filters, deductive systems and representation theorems} \label{S3}
	
\begin{definition} A subset $F$ of a Nelson algebra $N$ is a \emph{filter} if:
\begin{enumerate}[\normalfont (F1)]
	\item $1\in F$,\label{F1}
	\item if $x\in F, x\le y$ then $y\in F$,\label{F2}
	\item for all $x,y\in F$, $x\land y\in F$.\label{F3}
\end{enumerate}
A proper filter $P$ of a Nelson algebra $N$ is a \emph{prime filter} if $a\lor b\in P$ implies $a\in P$ or $b\in P$.
\end{definition}

\begin{definition}\label{RBR}	
Let $X(R)$ the set of all prime filters of a De Morgan algebra $R$. If $P\in X(R)$ let $\varphi(P)=\complement\sim P=\sim\complement P$ {\rm (Birula-Rasiowa transformation, {\rm H. Rasiowa \cite{bialynicki57quasiboole}})}. We know that $\varphi(P)\in {X}(R)$. Since every Nelson algebra $N$ is a Kleene algebra, it is well known that $P$ and $\varphi(P)$ are comparable {\rm \cite{cignoli86class}}.	
\end{definition}
	
\begin{definition}
A subset $D$ of a Nelson algebra $N$ is called a \emph{deductive system} if:
\begin{enumerate}[\normalfont (D1)]
	\item $1\in D$,\label{defds1}
	\item If $x,x\to y\in D$ then $y\in D$.\label{defds2}
\end{enumerate}
We will denote with $\mathcal{D}(N)$ the family of all deductive systems of $N$. A deductive system $D$ is \emph{proper} if $D\ne N$. All deductive systems are filters. 
	
We say that $D\in {\cal D}(N)$ is:
\begin{enumerate}[\normalfont(1)]
	\item \emph{Maximal}, if $D$ is proper, and there is no proper deductive system that properly contains $D$. We denote with $\mathcal{M}(N)$ the family of all maximal deductive systems of $N$.
	\item \emph{Irreducible}, if $D$ is proper, and if $D=D_1\cap D_2$ with $D_1,D_2 \in {\cal D}(N)$ then \linebreak{$D=D_1$} or $D=D_2$. We denote with ${\cal I}(N)$ the family of all irreducible deductive systems of $N$.
	\item \emph{Completely irreducible}, if $D$ is proper and if $D=\bigcap\limits_{i\in I} D_i$ with $D_i \in {\cal D}(N)$ for all $i\in I$ then there exists an index $i\in I$ such that $D=D_i$. We denote with ${\cal CI}(N)$ the family of all completely irreducible deductive systems of $N$.
	\item \emph{Bounded to an element} $d\ne 1$ if it is maximal among those that do not include $d$.
\end{enumerate}
\end{definition}

	\begin{remark}
	\label{Rnuev}
	From the definitions, it's clear that  ${\cal CI}(N)\subseteq {\cal I}(N).$
\end{remark}

	\begin{lemma} {\rm (H. Rasiowa, \cite{rasiowa74algebraic})}
	\label{L41}
	If $N$ is a non-trivial Nelson algebra, ${\cal M}(N)\subseteq{\cal I}(N).$ 
\end{lemma}

\begin{proof}
	Let $D$ be a maximal deductive system. If $D=D_1\cap D_2$, then $D\subseteq D_i$, and since $D$ is maximal, we must have $D=D_1=D_2$.
\end{proof}

If $X$ is an arbitrary subset of elements of a Nelson algebra, we denote with $D(X)$ the deductive system generated by $X$. This is, the least deductive system of $N$ containing $X$. We denote with $D(D,a)$ the deductive system generated by $D\cup \{a\}$. We will use the the following result that gives the form of the generated deductive system for some special cases (see also A.~Monteiro, \cite{AM80}, Viglizzo, \cite{viglizzo99algebras}, pp. 27 and 36).

\begin{lemma} {\rm (A. Monteiro, \cite{monteiro62algebra})} If $N$ is a Nelson algebra, $D$ is a deductive system and $a, b \in N$, 
	\label{L42}
	\begin{itemize}
		\item
  $D(D,a)=\{x\in N:a\to x\in D\}$, and
		\item $D(\set{a,b})=\{x\in N: (a \land b)\to x=1\}=\{x\in N: a \to ( b\to x)=1\}$. 
	\end{itemize}
\end{lemma}

	\begin{lemma}
	\label{L416} 
	{\rm (A. Monteiro, \cite{monteiro62algebra})}
	If $M \in {\cal M}(N)$ and $x\notin M$ then   $x\to y \in M$ for every $y\in N$.
\end{lemma}

\begin{proof} Let $M \in {\cal M}(N)$, $x \notin M$. Since $M$ is maximal, $D(M,x)=N$ and therefore by Lemma 	\ref{L42}, for any $y\in N$,  $x\to y\in M$.
\end{proof}

\begin{lemma} {\rm (A.~Monteiro, \cite{AM80,monteiro2017})}
	\label{L43res} 
	If $N$ is a  Nelson algebra, $D\in {\cal D}(N)$, $D\neq N$,  and $d\in D$, then  
	 $\sim d \notin D$ and  $D\cap \sim D=\emptyset$.
\end{lemma} 
\begin{proof}
	Assume that both $d$ and $\sim d$ are in $D$. Then $d\land\sim d\in D$ and by (N\ref{111220244}), \linebreak[4] $(d\land \sim d)\to y=1$ for all $y\in N$, so $D=N$, a contradiction. 
	
	To see the second part, consider some $d\in D\cap\sim D$. Then $d\in D$ and $d=\sim c$ form some $c\in D$, but if this is true, $c=\sim d$ is in $D$, and we just proved that this is not the case.  
\end{proof}

\begin{lemma}
	\label{L410} {\rm (A. Monteiro, \cite{AM80})}
	If $N$ is a Nelson algebra, and $D\in {\cal I}(N)$ then {$D\in X(N)$} and $D\subseteq\varphi(D)$.
\end{lemma}
\begin{proof} It is well known that every deductive system is a filter.  Let us prove that in this case $D$ is a prime filter.
	
	Suppose that $a\lor b\in D$ and let $D_1=D(D,a)=\{x\in N:a\to x \in D\}$, and $D_2=D(D,b)=\{x\in N:b\to x \in D\}$, so $D\subseteq D_1,D_2$ and therefore $D\subseteq D_1\cap D_2$.
	
	Let $z\in D_1\cap D_2$, then $a\to z\in D$ and $b\to z\in D$, therefore
	$(a\lor b)\to z{\stackrel{{\rm (N\ref{111220242})}}{=}}$ $(a\to z) \land (b\to z)\in D$ and since $a\lor b\in D$ it follows that $z\in D$. Hence $D=D_1\cap D_2$ and since $D\in{\cal I}(N)$, $D=D_1$ or $D=D_2$.  It follows that  $a\in D$ or $b\in D$, so $D\in X(N)$. Since $D$ is proper, then by Lemma \ref{L43res}, $D\cap\sim D=\emptyset$ then $D\subseteq \varphi(D)$. Indeed, if $d\in D$, then $\sim d\notin D$, so $d\in\varphi(D)$.
\end{proof}

\begin{lemma} {\rm (A. Monteiro, \cite{AM80})}
	\label{L49}
	If $N$ is a Nelson algebra, and $P\in X(N)$ verifies $P\subseteq\varphi(P)$ then $P\in{\cal I}(N)$.
\end{lemma}

\begin{proof} Since $P$ is a prime filter, it verifies (D\ref{defds1}). Let us prove that it verifies (D\ref{defds2}). If $x,x\to y \in P$, since $P$ is a filter, $x \land(\sim x\lor y) {\stackrel{{\rm (N8)}}{=}}x\land(x\to y)\in P$. Then  $\sim x\lor y\in P$, and since $P\in X(N)$ we have that $\sim x\in P$ or $y \in P$. If $\sim x \in P$ since $P\subseteq\varphi (P)$ then $\sim x\in\varphi(P)$, so $x\in\complement P$, a contradiction. Hence $y\in P$, so $P\in {\cal D}(N).$ Since $P\in X(N)$ then $P$ is proper. Suppose that $P=D_1\cap D_2$ with $D_1,D_2\in {\cal D}(N)$ and that $D_1\ne P$, $D_2\ne P$, then $P\subset D_1$ and $P\subset D_2$. Let $a\in D_1\setminus P$ and $b\in D_2\setminus P$. From this it follows that $a\lor b\in D_1\cap D_2=P$ and since $P\in X(N)$, then $a\in P$ or $b\in P$, contradiction.
\end{proof}

\begin{corollary}\label{equivalenciasP1}
	In a Nelson algebra $N$, given a subset $P\subseteq N$, 
	\[  \mathcal{I}(N)=\set{P\in X(N)|P\subseteq\varphi(P)}. \]
\end{corollary}

\begin{theorem} {\rm (H. Rasiowa, \cite{rasiowa74algebraic})}
	\label{THR}  \label{CIdeductivesystem}
	If  $N$ is a Nelson algebra, $D\in{\cal D}(N)$, $D\neq N$, and $a\notin D$, then there exists $D_a\in{\cal CI}(N)$ such that $D\subseteq D_a$ and $a\notin  D_a$.  {\rm (see also   I.Viglizzo, \cite{viglizzo99algebras}, page 40)} 
\end{theorem}

\begin{proof}
	By hypothesis, $D$ is a proper deductive system. Let 
	\[\mathcal{D}_a=\set{D'| D' \text{ is a deductive system}, D\subseteq D',\text{ and }a\notin D'}.\]
	It is easy to check that in $\mathcal{D}_a$ every chain has an upper bound, so by Zorn's Lemma, there exists a maximal element in $\mathcal{D}_a$, namely $D_a$. 
	
	Now we prove that $D_a$ is completely irreducible. If $D_a=\bigcap_{i\in I}D_i$ with $D_i\neq D_a$ for every $i\in I$, then $a\in D_i$ for all $i\in I$, so $a\in D_a$, a contradiction. 
\end{proof}

\begin{theorem} {\rm (A. Monteiro, \cite{monteiro62algebra})}
	\label{T41}
	A deductive system   $D$ is completely irreducible  if and only if  $D$ is bounded to some element $a$ {\rm (see also  A.~Monteiro, \cite{monteiro80heyting}, I. Viglizzo, \cite{viglizzo99algebras}, page 41)}.
\end{theorem}
\begin{proof}
	Let $D$ be a completely irreducible deductive system. By definition, $D$ is a proper deductive system. For each element $a\in N\setminus D$, there exists by Theorem \ref{CIdeductivesystem}, a bounded deductive system $D_a$ containing $D$ and such that $a\notin D_a$. Thus $D=\bigcap_{a \in N\setminus D}D_a$, and since $D$ is completely irreducible, $D=D_a$ for some $a$. 
	
	Now assume $D_a$ is a deductive system bounded to the element $a$. If $D_a=\bigcap_{i\in I}D_i$, then for some $i\in I$, $a\notin D_i$. We also have that $D_a\subseteq D_i$, but by the definition of bounded deductive system, we must have $D_a=D_i$, so $D_a$ is completely irreducible.
\end{proof}
\begin{definition} {\rm (A.~Monteiro, \cite{AM62a})} A Nelson algebra $N$, is said to be {\emph{linear}} if for every $x,y\in N$,  \[(x\to y)\lor (y\to x)=1\] holds.
\end{definition}
	\begin{lemma} \label{LAL}
	If $N$ is a five-valued Nelson algebra, then $N$ is a linear Nelson algebra. 
\end{lemma}

\begin{proof} By Theorem \ref{equivalenciasFN}, $1=1 \vee 1 {{\rm\stackrel{(FN1)}{=}}} y\vee\lceil x \vee (y\to x)\vee x\vee \lceil y \vee (x\to y)=$ \linebreak[4] $\lceil y \vee  x \vee (y\to x)\vee \lceil x \vee y \vee (x\to y){\stackrel{{\rm(N\ref{111220245})}}{=}}$	$\lceil y \vee (y\to x)\vee \lceil x  \vee (x\to y){{\rm\stackrel{(N\ref{N19})}{=}}}$ \linebreak[4] $(y\to x)\vee(y\to x).$
\end{proof}

If $N$ is a Nelson algebra and $I\in {\cal I}(N)$, let ${\cal S}(I)=\{D\in D(N): D\neq N, I\subseteq D\}$. Since $I\in {\cal S}(I)$, then ${\cal S}(I)\neq \emptyset$. Let ${\bf O}(I)$ be the poset $({\cal S}(I),\subseteq)$. 

\begin{lemma}  
	{\rm (A. Monteiro, \cite{monteiro62algebra})}
	\label{L45N}
	If $N$ is a linear Nelson algebra and $I\in {\cal I}(N)$, then ${\bf O}(I)$ is a chain. 
\end{lemma}
\begin{proof} 
	Suppose that ${\bf O}(I)$  is not a chain, so there exist $D_1,D_2\in {\cal S}(I)$, $D_1\neq D_2$ such that $D_1\not\subseteq D_2$ and  $D_2 \not\subseteq D_1$. Therefore, there exist $a,b\in N$ such that 	(1) $a \in D_1\setminus D_2$ and (2) $b \in D_2\setminus D_1$. Since $N$ is linear, 
	$(a\to b)\lor (b \to a)=1$. As $1\in I$, and by Lemma \ref{L410} $I$ is a prime filter,  (3) $a\to b\in I\subset D_1$ or (4) $b\to a \in I\subset D_2$.
	From (1) and (3) it follows that $b\in D_1$, contradicting (2). Similarly, from  (2) and (4) it follows that  $a\in D_2$, contradicting (1).
\end{proof}

\begin{corollary}  
	\label{C45N} {\rm (D. Brignole, \cite{brignole65nelson}, Theorem 3, pag. 3)} 
	If $N$ is a five-valued Nelson algebra and $I\in {\cal I}(N)$, ${\bf O}(I)$ is a chain. 
\end{corollary}
\begin{proof} It follows immediately from lemmas \ref{LAL} and \ref{L45N}.
\end{proof}

	\begin{lemma} {\rm (A. Monteiro, \cite{monteiro62algebra})}
	\label{L47}
	Every proper deductive system of a Nelson algebra can be written as the  intersection of completely irreducible deductive systems.
\end{lemma}
\begin{proof}
	Is a consequence of the proof of \ref{T41}.
\end{proof}

\begin{lemma}
	\label{CT42} If $D$ is a deductive system bounded to $d$ then $x\to d \in D$ for every $x\notin D$.
	{\rm (see  A. Monteiro, \cite{monteiro80heyting}, I.Viglizzo, \cite{viglizzo99algebras}, page 42, Corollary 2.4)}
\end{lemma}
\begin{proof}
	Consider the deductive system generated by $D$ and $x$ for any $x\notin D$. By Lemma \ref{L42} this is $D(D,x)=\set{z\in N:x\to z\in D}$. Given that $x\notin D$, $D$ is properly contained in $D(D,x)$, and since $D$ is bounded to $d$, $d\in D(D,x)$, so $x\to d\in D$.
\end{proof}

\begin{theorem} {\rm (D. Brignole, \cite{brignole65nelson}, Theorem 1, pag. 2)}
	\label{T43}
	If $N$ is a five-valued Nelson algebra,  $C\in {\cal CI}(N)$, $D\in {\cal D}(N)$ is proper and  $C\subset D$, then $D\in {\cal M}(N)$. 
\end{theorem}

\begin{proof} Since $C\in {\cal CI}(N)$ then by Theorem \ref{T41}, $C$ is bounded to some $b\neq 1$. As $C\subset D$ then $b\in D$. Assume there exists $M\in{\cal D}(N)$ such that $D\subset M\subseteq N$.
	Let $a\in M\setminus D$. If $b\to a\in D$, since  $b\in D$, then $a \in D$, a contradiction. Then
	$b\to a\notin D$
	and therefore
	$b\to a\notin C,$ so by  Lemma \ref{CT42}, 
	\begin{equation} \label{201124_T43_02} 
		(b\to a)\to b \in C.	
	\end{equation}	
	If 	$((b\to a)\to b)\to b \in C$ then by (\ref{201124_T43_02}) it follows that $b\in C$, a contradiction. Then
	\begin{equation} \label{201124_T43_03} 
		((b\to a)\to b)\to b \notin C.	
	\end{equation}	
	If there exists some $c\in N\setminus M$ then 
	\begin{equation} \label{201124_T43_04} 
		a\to c \notin M,	
	\end{equation}
	for otherwise, if $a\to c \in M$,  since $a\in M$, it would follow that  $c \in M$.
	From (\ref{201124_T43_04}) it follows that $a\to c \notin C$, so by  Lemma \ref{CT42}, 
	\begin{equation} \label{201124_T43_06} 
		(a\to c)\to b \in C.	
	\end{equation}
	Since  $N$ is a five-valued algebra  then
	$$ ((a\to c)\to b)\to(((b \to a)\to b)\to b)=1\in C$$
	so by (\ref{201124_T43_06}), $((b \to a)\to b)\to b \in C$, which contradicts (\ref{201124_T43_03}). Therefore, $M=N,$ and $D$ is maximal.
\end{proof}

	\begin{corollary} {\rm (D. Brignole, \cite{brignole65nelson}, Corollary 1, pag. 2)}
	\label{C42} If  $N$ is a five-valued Nelson algebra and $C\in {\cal CI}(N)$ is not maximal, then there  exists a unique maximal deductive system that properly contains $C$.
\end{corollary}
\begin{proof} Since $C$ is not maximal, there exists a proper deductive system $D$ such that   $C\subset D$. By Theorem \ref{T43}, $D$ is maximal.
	
	Now assume there exist   $D_1,D_2 \in {\cal M}(N)$,   such that $C\subset D_1,D_2$. By the observation \ref{Rnuev}, $C\in {\cal I}(N)$  so by Corollary \ref{C45N}, $D_1\subseteq D_2$	 or $D_2\subseteq D_1$ and therefore by  maximality $D_1=D_2.$
\end{proof}

	\begin{lemma} {\rm  (D. Brignole, \cite{brignole65nelson}, Theorem 4, pag. 3)}
	\label{L48} If $N$ is a five-valued Nelson algebra then ${\cal I}(N) \subseteq {\cal CI}(N)$.
\end{lemma}
\begin{proof} Let  $I\in {\cal I}(N)$, so $I$ is proper and by  Lemma \ref{L47}, $I=\bigcap\limits_{j\in J} C_j$, with $C_j\in {\cal CI}(N)$, for every $j\in J$.
	Then $I\subseteq C_j$ for every $j\in J$. By Corollary \ref{C45N}, ${\cal K}=\{C_j\}_{j\in J}$  is a chain. This chain has at most two elements: if $C_1,C_2,C_3$ are three different elements in ${\cal K}$  such that $C_1\subset C_2\subset C_3$, then the completely irreducible deductive system  $C_1$ would be properly contained in two proper deductive systems, which contradicts Theorem \ref{T43}. Then ${\cal K}$ has at most two elements, $C_1,C_2$ and since  $I=C_1\cap C_2$, by hypothesis $I=C_1$ or $I=C_2$ and therefore $I\in {\cal CI}(N)$.
\end{proof}

\begin{corollary}
	\label{C48a}  If $N$ is a five-valued Nelson algebra then ${\cal I}(N)={\cal CI}(N)$.
\end{corollary}

\begin{theorem}{\rm (D. Brignole, \cite{brignole65nelson}, Theorem 2, pag. 2)}
	\label{T44}
	If $N$ is a  Nelson algebra in which every completely irreducible deductive system is maximal or there exists a unique proper deductive system containing it properly, then $N$ is five-valued.
\end{theorem}

\begin{proof} Assume $N$ is as in the statement of the theorem, but not five-valued, that is, there exist $a,b,c\in N$ such that 
	\begin{equation} \label{201124_T44_01}
		((a\to c)\to b)\to(((b \to a)\to b)\to b)\neq 1.	
	\end{equation}	
	Let $D$ be the deductive system generated by the elements $(a\to c)\to b$ and $(b \to a)\to b$, so by Lemma \ref{L42}
	$$ D=\{x \in N: ((a\to c)\to b)\to (((b \to a)\to b)\to x)=1\}.$$
	So by (\ref{201124_T44_01}),	 
	$b\notin D$ and therefore by Theorem \ref{THR} there exists $C\in {\cal CI}(N)$ such that 
	\begin{equation} \label{201124_T44_03}
		b\notin C\ \mbox{ and } D\subseteq C.	
	\end{equation}
	Since $(b \to a)\to b\in D\subseteq C$, if  $b\to a \in C$ then $b\in C$, a contradiction. Therefore, 
	\begin{equation} \label{201124_T44_04}
		b\to a\notin C	
	\end{equation}
	and since $a\leq b\to a$ we also have  
	\begin{equation} \label{201124_T44_05}
		a\notin C.	
	\end{equation}
	Let  $D_1=D(C, b)=\{x\in N:b\to x \in C\}$, and $D_2=D(C, a)=\{x\in N:a\to x \in C\}$,	so   $C\subseteq D_1$ and $C\subseteq D_2$.
	By (\ref{201124_T44_04}) $a\notin D_1$, so $D_1$ is a proper deductive system and since $b\in D_1$, by (\ref{201124_T44_03}) it follows that 
	\begin{equation} \label{201124_T44_06}
		C \subset D_1.	
	\end{equation}
 Since $(a\to c)\to b\in D\subseteq C$, if $a\to c \in C$ then $b\in C$, a contradiction, so 	$a\to c \notin C$. Therefore $c\notin D_2$,  $D_2$ is proper and since $a\in D_2$, by (\ref{201124_T44_05}) we obtain 
	\begin{equation} \label{201124_T44_07}
		C\subset D_2.	
	\end{equation}	
	From $a\in D_2$ and $a\notin D_1$, we conclude that $D_1\neq D_2$.
	
	From (\ref{201124_T44_06}) or (\ref{201124_T44_07}) it follows that $C$ is not  a maximal deductive system and both   $D_1$ and $D_2$ are proper deductive systems properly containing  $C$, which contradicts the hypothesis, so $N$ is five-valued.
\end{proof}

From  Corollary \ref{C42} and Theorem \ref{T44} we get the two following corollaries that characterize five-valued Nelson algebras in terms of their irreducible deductive systems:

\begin{corollary} {\rm (D. Brignole, \cite{brignole65nelson}, Corollary 2, pag. 2)}
	\label{Cdb}  A Nelson algebra is five-valued if and only if every non-maximal $C\in {\cal CI}(N)$, is properly contained in a unique proper deductive system.
\end{corollary}

\begin{corollary} {\rm (D. Brignole, \cite{brignole65nelson}, Corollary 3, pag. 3)}
	\label{C42a}  A Nelson algebra $N$ is five-valued if and only if for every  $D\in {\cal I}(N)$, $D\notin {\cal M}(N)$ or there exists a unique $M \in {\cal M}(N)$ such that $D\subset M$.
\end{corollary}

\begin{lemma}
	\label{L51} {\rm (A. Monteiro, \cite{monteiro63nelson})}.
	If $N$ is a Nelson algebra and  $D$ a deductive system in $N$ then the relation defined on  $N$ by 
	$$x\equiv_D y\text{ if and only if } x\to y, y\to x, \sim x\to \sim y,\sim y\to \sim x \in D$$ is a congruence of Nelson algebras. 
\end{lemma}

If $x\in N$, we denote the equivalence classes by $|x|_D=\{y \in N:y  \equiv_D x\}$. We write  $N/D$ for the quotient algebra  $N$ over $D$ determined by the equivalence classes of the relation  $\equiv_D$, with the induced operations. The homomorphism $h_D: N \to N/D$ defined by $h_D(x)=|x|_D$ is called the natural or canonical epimorphism  from $N$ onto $N/D$.  It is well known that $|1|_D=D$, $|0|_D=\sim D$, and  
\begin{equation}
	\label{ec1}
	|0|_D\leq |x|_D\leq  |1|_D,\text{ for every } x\in N.
\end{equation}

\subsection{Representation theorems}

\begin{lemma}  \label{lema_pre_teo_repres2}
	Let	$N$ be a Nelson algebra and $M \in{\cal M}(N)$. Then:
	\begin{enumerate}[{\rm(a)}]
		\item If $M = \varphi(M)$ then $N/M\cong C_2$ and the corresponding equivalence classes are  $M$ and $\sim M$.
		\item If $M\subset \varphi(M)$ then $N/M\cong C_3$ and the equivalence classes of the quotient are  $M,\varphi(M)\setminus M$, and $\sim M$.
	\end{enumerate}
\end{lemma}

\begin{proof}
	\begin{enumerate}[(a)]
		\item From  $M = \varphi(M)$, it follows that $\sim M=\complement M$, so $\set{M,\sim M}$ is a partition of $N$. Furthermore,  $|0|_M=\complement M \neq M= |1|_M$, and by  	(\ref{ec1}), $|0|_M< |1|_M$ so	$N/M\cong C_2$.
		
		\item Since  $\varphi(M)\in X( N)$, then $\varphi(M) \subset N$, so we have the proper inclusions \linebreak[5] $M \subset \varphi(M) \subset N$. 	Let us now prove that
		\begin{equation} \label{151124_12}
			\text{if } x\in \varphi(M)\setminus M   \text{ then } |x|_M=\varphi(M)\setminus M.
		\end{equation}
		Let 
		\begin{equation} \label{151124_13}
			x\in \varphi(M)\setminus M.
		\end{equation}
		From (\ref{151124_13}) we get that 
		\begin{equation} \label{151124_14}
			x,\sim x \notin M.
		\end{equation}
		We will check that $|x|_M=\varphi(M)\setminus M.$ Consider first $t\in |x|_M$ so in particular we have 
		\begin{equation} \label{151124_15}
			t\to x\in M
		\end{equation}
		and
		\begin{equation} \label{151124_16}
			\sim t \to \sim x\in M.
		\end{equation}
		If $t\in M$ then from (\ref{151124_15}) it follows $x\in M$, contradicting (\ref{151124_14}). If $t\notin \varphi(M)$ then $\sim t\in M$ so by (\ref{151124_16}) we get $\sim x \in M$, contradicting (\ref{151124_14}). Therefore, $t \in \varphi(M)\setminus M $ so we have proved that $|x|_M \subseteq \varphi(M)\setminus M$.
		
		To see the other inclusion, let $t \in \varphi(M)\setminus M$. Then 
		\begin{equation} \label{151124_17}
			t,\sim t \notin M.
		\end{equation}
		From  (\ref{151124_14}) and (\ref{151124_17}), since $M$ is maximal,  by Lemma \ref{L416},  we obtain $x\to t, t\to x$, \linebreak[3] $\sim x\to \sim t, \sim t \to \sim x \in M$, so by Lemma \ref{L51} $t\in |x|_M$. Therefore,  $\varphi(M)\setminus M\subseteq |x|_M$, and condition (\ref{151124_12}) is proved.
		
		Next we prove:
		\begin{equation} \label{151124_18}
			\text{if } x\in \varphi(M)\setminus M \text{ then } \sim |x|_M= |x|_M.
		\end{equation}
		Let $x\in \varphi(M)\setminus M$.
		From (\ref{151124_12}), $|x|_M= \varphi(M)\setminus M$, so $\sim|x|_M= \sim(\varphi(M)\setminus M)=$ \linebreak[3] $\sim (\varphi(M)\cap \complement M) \stackrel{{\rm(N\ref{111220249})}}{=}\sim \sim \complement M \cap \sim \complement M= \complement M\cap  \varphi(M)= \varphi(M)\setminus M=|x|_M$.
		
		Every element not in $M$ or $\sim M$ is in $\varphi(M)\setminus M$: $\complement (M\cup\sim M)=\complement M\cap\complement\sim M=\varphi(M)\setminus M$. Thus,  $\{M,\varphi(M)\setminus M,\sim M\}$ is a partition of $N$. Therefore, $N/M$ is a Nelson algebra with three elements, and by  Observation \ref{R51},  $N/M\cong C_3$.
	\end{enumerate}	
\end{proof}

\begin{theorem} {\rm (A. Monteiro, \cite{monteiro62algebra,monteiro63nelson,monteiro95nelson,monteiro96nelson})}
	\label{T51}
	If $N$ is a Nelson algebra  and $M\in {\cal M}(N)$  then $N/M\cong C_2$ or $N/M\cong C_3$. 
\end{theorem}
\begin{proof}
	We know that $M\subseteq \varphi(M)$. Then by Lemma \ref{lema_pre_teo_repres}, $N/M\cong C_2$ or $N/M\cong C_3$.
\end{proof}

	\begin{lemma} 
	\label{L99} If $N$ is a five-valued Nelson algebra, $D\in {\cal I}(N)$,  $D\subset M$ and $ M\in {\cal M}(N)$,
	then $D$ is bounded to every $x\in M\setminus D$.  
\end{lemma}
\begin{proof} Let $x\in M\setminus D$ and assume $D$ is not bounded to  $x$. Then there exists $D'\in {\cal D} (N) $ such that $D\subset D'$ and $x\notin D'$.  Thus we have   $D\subset D'$ and $D\subset M$, with both $D'$ and $M$ proper deductive systems. Since $D$ is not maximal, by  Corollaries  \ref{C48a} 	 and \ref{C42} it follows that $D'=M$,  which is a contradiction because $x\notin D'$ and $x\in M$.
\end{proof}

	\begin{corollary}
	\label{C99}
	If $N$ is a five-valued  Nelson algebra, $D\in {\cal I}(N)$, $D\subset M$, and \linebreak[5] $ M\in {\cal M}(N)$ then
	$y\to x \in D$ for every $x \in  M\setminus D$, $y\notin D$.  
\end{corollary}
\begin{proof} By Lemma \ref{L99} $D$ is bounded to each $x\in M\setminus D$ so by Lemma \ref{CT42}, $y\to x\in D$ for every $y \notin D$.           
\end{proof}

\begin{lemma} \label{lema_pre_teo_repres}
	If	$N$ is a non-trivial Nelson algebra, $D\in{\cal I}(N)$, and $M \in{\cal M}(N)$ are such  $D \subset M$, then:
	\begin{enumerate}[{\rm (a)}]
		\item If $M = \varphi(M)$ then $N/D\cong C_4$, and the equivalences classes are  $D, M\setminus D, \varphi(D)\setminus M$, and $\sim D$.  Furthermore, $N/M\cong C_2$.
		\item If $M\subset \varphi(M)$ then $N/D\cong C_5$, with equivalence classes $D, M\setminus D, \varphi(M)\setminus M$, \linebreak[4] $\varphi(D)\setminus \varphi(M)$, and $\sim D$. In this case $N/M\cong C_3$.
	\end{enumerate}
\end{lemma}

\begin{proof}
	\begin{enumerate}[(a)]
		\item Since  $D \subset M$,  $\sim D \subset \sim M$, and $\varphi(M) \subset \varphi(D).$ From $M=\varphi(M)$ it follows that 			
		\begin{equation}
			\label{ec1aa}
			\sim M=\complement M
		\end{equation}
		and  $D\subset M=\varphi(M)\subset \varphi(D)\subset N$.
		
		Now we check that the sets $(\varphi(D)\setminus M)$ and $\sim D$ are disjoint:
		$$(\varphi(D)\setminus M)\cap \sim D=
		\varphi(D)\cap \complement M \cap \sim D\stackrel{(\ref{ec1aa})}{=}$$
		$$ \varphi(D)\cap \sim M\cap \sim D\stackrel{{\rm (N\ref{111220249})}}{=}\varphi(D)\cap \sim(M\cap D)=\varphi(D)\cap \sim D=\emptyset.$$
		Next we prove: 
		\begin{equation} \label{151124_07}
			\mbox{if } z \in M\setminus D   \text{ then } |z|_{D}= M\setminus D.
		\end{equation}
		Let 
		\begin{equation} \label{151124_03}
			z \in M\setminus D.
		\end{equation}
		Consider  $t\in M\setminus D$. Since $t\in M\setminus D$ and $z \in M\setminus D$, by Corollary \ref{C99}, 
		\begin{equation} \label{151124_01}
			t\to z, z\to t\in D.
		\end{equation}
		Since $M$ is a proper deductive system and  $t,z\in M$, by Lemma \ref{L43res}, $\sim t,\sim z \notin M$, so by Lemma \ref{L416},  $\sim t\to \sim z,\sim z\to \sim t\in M$.
		If we assume $\sim z \to \sim t \notin D$, this is, $\sim z \to \sim t \in M\setminus D$, then since  $\sim z \notin D$, it follows by  Corollary \ref{C99} that  $\sim z\to \sim t{\stackrel{{\rm (N\ref{1112202410})}}{=}} \sim z\to (\sim z\to \sim t)\in D$, a contradiction. 
		Therefore
		\begin{equation}  \label{151124_02}
			\sim z \to \sim t\in D \text{  and similarly } \sim t \to \sim z \in D.
		\end{equation}
		From conditions (\ref{151124_01}) and (\ref{151124_02})
		it follows that  $t\in |z|_{D}$, so $M\setminus D\subseteq |z|_{D}.$
		
		For the other inclusion,  consider  $t\in |z|_{D}$. Then we have 
		\begin{equation} \label{151124_04}
			t\to z \in D
		\end{equation}
		and
		\begin{equation} \label{151124_05}
			z\to t\in D.
		\end{equation}
		If $t\in D$ then by (\ref{151124_04}) $z\in D$, which contradicts (\ref{151124_03}). Then $t\notin D$.
		From  (\ref{151124_05}) and $D\subset M$ it follows  $z\to t \in M$, and since  $z\in M$ we have that $t\in M$, so $t\in M\setminus D$. Therefore $|z|_{D}\subseteq M\setminus D$.
		
		Thus  $|z|_{D}= M\setminus D$,  proving condition (\ref{151124_07}).
		
		Let us prove now that 
		\begin{equation} \label{151124_08}
			\text{if }  x \in \varphi(D)\setminus M \text{ then } |x|_{D}= \varphi(D)\setminus M.
		\end{equation}
		Let 	$x \in \varphi(D)\setminus M$. Then $\sim x \in \sim(\varphi(D)\setminus M)= \sim(\varphi(D)\cap \complement M)\stackrel{(\ref{ec1aa})}{=}$ \linebreak[5] $\sim (\varphi(D)\cap \sim M)\stackrel{{\rm(N\ref{111220249})}}{=}\sim \varphi(D)\cap M=
		\complement D\cap M= M\setminus D$. 	 By condition (\ref{151124_07}), $M\setminus D=|y|_{D}$ for $y\in M\setminus D$. Therefore $\sim |x|_{D}=|\sim x|_{D}=|y|_{D}$, and   $|x|_{D}=\sim|y|_{D}=\sim (M\setminus D)=
		\varphi(D)\setminus M$ proving (\ref{151124_08}).
		
		Using conditions (\ref{151124_07}) and (\ref{151124_08}) we can deduce that  $$N/D=\{D, M\setminus D, \varphi(D)\setminus M,\sim D\}$$ is a Nelson algebra with four elements. Let us check it is isomorphic to the chain $C_4$, and not  a  boolean algebra. Let
		\begin{equation} \label{151124_10}
			x\in \varphi(D)\setminus M  \text{ and }
			y\in M\setminus D.
		\end{equation}
	
		Assume  $x\land y \equiv_D 0$. This implies that
		$\sim x \lor \sim y= \sim (x \land y)=\sim 0\to \sim (x \land y)\in D$, and since by Lemma \ref{L410}, $D\in X(N)$, then $\sim x \in D$ or  $\sim y \in D$. If we had $\sim x \in D$, then  $ x \in \sim D$ and therefore $x \notin \complement \sim D=\varphi(D)$, contradicting (\ref{151124_10}). If  $\sim y \in D$ instead, since $D\subset M$, we get $\sim y \in M$ and thus $ y \in \sim M \stackrel{(\ref{ec1aa})}{=} \complement M$, also contradicting (\ref{151124_10}). It follows then that $|x\land y|_D\neq |0|_D$.
		
		Thus, by  Observation \ref{R51}, $N/D\cong C_4$.

		Observe that from  $M=\varphi(M)$ and $M \in{\cal M}(N)$ then by  Lemma \ref{lema_pre_teo_repres2} we also have $N/M\cong C_2$.
		
		\item As in the first part, $\varphi(M) \subset \varphi(D)$. We have the chain of proper inclusions: \linebreak[5] $D\subset M \subset \varphi(M)\subset \varphi(D)\subset N$. Since the complement of $\varphi(D)$ is $\sim D$, from the inclusions it is easy to see that $\{D, M\setminus D, \varphi(M)\setminus M, \varphi(D)\setminus \varphi(M),\sim D\}$  is a partition of $N$.
		
		Just as in the previous case, we can prove
		\begin{equation} \label{191124_01}
			\mbox{if } z \in M\setminus D \  \mbox{then } |z|_{D}= M\setminus D.
		\end{equation}
		Let us see now that 
		\begin{equation} \label{191124_02}
			\text{if } y \in \varphi(M)\setminus M  \text{ then } |y|_{D}= \varphi(M)\setminus M.
		\end{equation}
		Let
		\begin{equation} \label{191124_03}
			y \in \varphi(M)\setminus M.
		\end{equation}
We want to prove that  $|y|_{D}= \varphi(M)\setminus M$. Consider
		\begin{equation} \label{191124_04}
			b \in \varphi(M)\setminus M.
		\end{equation} 
		From (\ref{191124_03}) and (\ref{191124_04}) we get 
		\begin{equation} \label{191124_05}
			b,y \notin M
		\end{equation} 
		and therefore
		\begin{equation} \label{191124_06}
			b,y \notin D.
		\end{equation} 
		From (\ref{191124_05}) and Lemma \ref{L416} it follows that $b\to y, y \to b\in M$. If $b\to y, y \to b\notin D$, from (\ref{191124_06}) and Corollary \ref{C99}, using (N\ref{1112202410}),  $b\to y=b\to (b\to y)\in D$ and
		$y\to b=y\to (y\to b)\in D$, a contradiction. Then $b\to y$, $y \to b\in D$.
		Also, from (\ref{191124_03}) and (\ref{191124_04}) it follows that  $b,y \notin \sim M$, so $\sim b,\sim y \notin  M$ and with the same reasoning,  $\sim b\to \sim  y, \sim y\to \sim b \in D$. Therefore $b\in |y|_{D}$, and as a consequence, $\varphi(M)\setminus M\subseteq |y|_{D}.$ 
		
For the other inclusion, assume $b \in |y|_{D}$.  Then we have $b\to y, \sim b\to \sim y \in D$ and since $D\subset M$ then 
		\begin{equation} \label{191124_07}
			b\to y \in M
		\end{equation} 
		and 
		\begin{equation} \label{191124_08}
			\sim b\to \sim y \in M.
		\end{equation} 
		If $b\notin \varphi(M)$ then $b \in \sim M$ and therefore
		\begin{equation} \label{191124_09}
			\sim b \in M.
		\end{equation} 
		From (\ref{191124_09}) and (\ref{191124_08}) it follows that $\sim y \in M$, so $y \notin \varphi(M)$, a contradiction. Then  $b\in \varphi(M)$.
		If $b\in M$, then by (\ref{191124_07}) $y \in M$, a contradiction. Then  $b\in \varphi(M)\setminus M$  and therefore $|y|_{D} \subseteq \varphi(M)\setminus M$.
		
		From the preceding, $|y|_{D}= \varphi(M)\setminus M$ verifying condition (\ref{191124_02}).
		
		We also have in this case that $\sim(\varphi(M)\setminus M)=
		\sim(\varphi(M)\cap \complement M)\stackrel{{\rm(N\ref{111220249})}}{=} \sim \varphi(M) \cap \varphi(M)= \complement M\cap \varphi(M)= \varphi(M)\setminus M$, so $\sim|y|_{D} = |y|_{D}$ when $y \in \varphi(M)\setminus M$ using (\ref{191124_02}).
		
		Let us prove now 
		\begin{equation} \label{191124_10}
			\text{if } x \in \varphi(D)\setminus \varphi(M)  \text{ then } |x|_D= \varphi(D)\setminus \varphi(M).
		\end{equation}
		Let 
		\begin{equation} \label{191124_11}
			x \in \varphi(D)\setminus \varphi(M).
		\end{equation}
		From (\ref{191124_11}) it follows $\sim x \in \sim(\varphi(D)\setminus \varphi(M))=M\setminus D=|z|_D$, where $z\in M\setminus D$, so
		$\sim |x|_D=|\sim x|_D=|z|_D$ and therefore $|x|_D=\sim|z|_D=$ $\sim (M\setminus D)\stackrel{{\rm(N\ref{111220249})}}{=}$ \linebreak[4] $\sim M\cap \sim\complement D=
		\sim M\cap \varphi(D) = \varphi(D)\setminus \varphi(M).$
		
		Thus we have that $$N/D=\{D, M\setminus D, \varphi(M)\setminus M, \varphi(D)\setminus \varphi(M),\sim D\}$$ is a Nelson algebra with five  elements, so by Observation \ref{R51},  $N/D \cong C_5.$ 	Notice that since $M\subset \varphi(M)$ and $M \in{\cal M}(N)$ then by Lemma \ref{lema_pre_teo_repres2}  we have $N/M\cong C_3$. 
	\end{enumerate}
\end{proof}

\begin{theorem}
	\label{T51a}  {\rm (D. Brignole, \cite{brignole65nelson}, Theorem 5, pag. 3)}
	If $N$ is a non-trivial five-valued Nelson algebra, and $D\in{\cal I}(N)$,  then $N/D \cong C_i$, for some $i$, $2\leq i\leq 5$.
\end{theorem}

\begin{proof}
	We proceed by cases on the deductive system $D\in{\cal I}(N)$.
	Assume first that $D\in{\cal M}(N)$.  By Theorem \ref{T51}, we know that in this case, $N/D \cong C_2$ or $N/D\cong C_3$. 
	
	If $D \notin {\cal M}(N)$, by Corollary \ref{C42a}, there exists a unique  $M \in{\cal M}(N)$ such that $D \subset M$. By Lemma \ref {L410}, $M \subseteq \varphi(M)$, so by Lemma \ref{lema_pre_teo_repres}, 	$N/D \cong C_4$ or $N/D\cong C_5$.
\end{proof}

\begin{definition}
	\label{Dsep} If $N$ is a Nelson  algebra, a set $\{D_i\}_{i\in I}\subseteq {\cal D}(N)$ is said to be {\emph separating family} if $\bigcap\limits_{i\in I} D_i=\{1\}$. 
\end{definition}

\begin{theorem}\label{representacion}
Let  $N$ be a non-trivial Nelson algebra, and let  $\mathcal{S}$  be a separating family of deductive systems of  $N$. Then $N$ is isomorphic to a subalgebra of the product $\prod_{D\in\mathcal{S}}N/D$. 
\end{theorem}

\begin{proof}
	We define a function $\psi: N\to \prod_{D\in\mathcal{S}}N/D$ by  
	$$\psi(x)(D)=h_D(x) \text{ for every } D\in {\cal S} \text{ and }x\in N.$$
	
	 We prove that  $\psi$ is an injective homomorphism.  
	 For every  $f,g\in  N$, and $D\in\mathcal{S}$ we have  $\psi(f \lor g)(D)= h_D(f\lor g)=h_D(f)\lor h_D(g)=\psi(f)(D)\lor \psi(g)(D)$. Similarly, we can prove that  $\psi(f \land g)= \psi(f)\land \psi(g),$
	$\psi(f \to g)= \psi(f)\to \psi(g),$ $\psi(\sim f)= \sim\psi(f)$, and $\psi(1)={\bf 1}$.
	
	Now we prove that  $\psi$ is injective. Indeed, if $\psi(f)=\psi(g)$ then
	$$\psi(f \to g)= \psi(f)\to \psi(g)={\bf 1},\; \psi(g \to f)= \psi(g)\to \psi(f)={\bf 1},$$
	$$\psi(\sim f \to \sim g)= \sim \psi(f)\to \sim\psi(g)={\bf 1},\; \psi(\sim g \to \sim f)= \sim \psi(g)\to \sim\psi(f)={\bf 1}.$$
	Therefore $f \to g, g \to f, \sim f \to \sim g, \sim g \to \sim f \in D$ for each $D \in {\cal S}$, and since  $\mathcal{S}$ is separating, we have 
	$$f \to g= g \to f= \sim f \to \sim g= \sim g \to \sim f =1$$
	so by (N\ref{1112202411}), $f=g$. Thus, $N\cong \psi(N)\subseteq \prod_{D\in\mathcal{S}}N/D$.
\end{proof}

\begin{theorem}
	A Nelson algebra $N$ is five-valued if and only if for every irreducible deductive system irreducible $D\in \mathcal{I}(N)$, $N/D$ is isomorphic to  $C_i$ for some $2\le i\le 5$.
\end{theorem}

\begin{proof} 
	If the Nelson algebra $N$ is five-valued and  $D\in\mathcal{I}(N)$, by Theorem \ref{T51a}, we know that  $N/D\cong C_i$, for some $2\leq i\leq 5$.
	
	For the converse, assume that for every irreducible deductive system $D$ in $N$ we have that   $N/D$ is  isomorphic to  $C_i$ for some $2\le i\le 5$. Since $C_5$ is a five-valued Nelson algebra, all of its subalgebras, the product of them, and any subalgebra of this product, are also five-valued. By Lemma \ref{L47} and Observation \ref{Rnuev},   ${\cal I}(N)$ is a separating family. 	If $N$ is a non-trivial  Nelson algebra,  then by  Theorem \ref{THR}, ${\cal I}(N)\neq \emptyset$. Thus, by Theorem \ref{representacion}, $N$ is isomorphic to a subalgebra of the product $\prod_{D\in\mathcal{I}(N)}N/D$, and each one of the factors in the product is a five-valued Nelson algebra. Therefore, $N$ is five-valued as well.
\end{proof}

Let ${\cal F}(N)= C_5^{{\cal I}(N)}$ be the set of all the functions from  ${\cal I}(N)$ to $C_5$, with the operations defined pointwise. ${\cal F}(N)$ is a five-valued Nelson algebra, with the top element ${\bf 1}$ defined as ${\bf 1}(D)=1$ for every $D\in {\cal I}(N)$.  Notice that we can assume that for each $D\in {\cal I}(N)$, $h_{D}$ is a homomorphism from $N$ to $C_5$. 
Thus $N$ is isomorphic to a subalgebra of the  product $\prod_{D\in\mathcal{I}(N)}N/D$ (see the previous proof), which can be regarded as a  subalgebra of ${\cal F}(N)= C_5^{{\cal I}(N)}$.
The following result obtains.
\begin{theorem} {\rm (Representation Theorem)}
	{\rm (D. Brignole, \cite{brignole65nelson}, Theorem 6, pag. 5)}
	\label{TTRep}  If $N$ is a non-trivial five-valued Nelson algebra, then $N$ is isomorphic to the subalgebra $N_1=\{\psi(f)\}_{f \in N}$ of ${\cal F}(N)=C_5^{{\cal I}(N)}$.
\end{theorem}

It is well known that:
\begin{lemma}
	\label{L61}
	If $N$ and $N'$ are Nelson algebras, $h:N \to N'$ an homomorphism, $G\subseteq N$ such that $S(G) = N$, then $S(h(G)) =h(N)$. That is, if $G$ generates $N$ then $h(G)$ generates $h(N)$. If $h$ is an epimorphism, then $S(h(G))=N'$.
\end{lemma}

\begin{definition}
	Given an homomorphism of Nelson algebras $h:N\to N'$, the \emph{kernel} of $h$ is the set $Ker(h)=\set{x\in N: h(x)=1}$. This set is always a deductive system. 
\end{definition}

\begin{lemma}
	\label{L62} 
	If $N$ and $N'$ are Nelson algebras and $h:N\to N'$ is an epimorphism, then $N'\cong N/Ker(h)$.
\end{lemma}

\begin{lemma}
	\label{LM1}
	If $N$ is a Nelson algebra and $h:N \to C_i$, $2\leq i \leq 5$ is an epimorphism, then $D=Ker(h)=h^{-1}(\{1\})\in {\cal I}(N)$. 
\end{lemma}
\begin{proof} Since $N$ and $C_i$ are bounded distributive lattices, then it is well known that since $\{1\}\in X(C_i)$
	and $h$ is in particular a lattice epimorphism, then (1) $D=Ker(h)\in X(N)$. Let us see that (2) $D\subseteq \varphi(D)$. Indeed, let $x\in D$ that is $h(x)=1$ and suppose that $x\notin \varphi(D)$, then $x\in \sim D$
	and therefore $\sim x \in D$, so $1=h(\sim x)=\sim h(x)=\sim 1=0$, a contradiction. From (1) and (2) it follows by Lemma \ref{L49}, that $D\in {\cal I}(N)$.
\end{proof}

\begin{lemma} \label{Lhom1}
	If $N,N_1,N_2$ are Nelson algebras, $h_i$ an epimorphism from $N$ to $N_i$, $i=1,2$ and $Ker (h_1) \subseteq Ker (h_2)$, then there exists an epimorphism $h$ from $N_1$ to $N_2$ such that $h\circ h_1=h_2$. If $Ker (h_1)= Ker ( h_2)$ then $N_1$ is isomorphic to $N_2$.
\end{lemma}
\begin{proof}
	Let $y\in N_1$. Since $h_1$ is an epimorphism, there exists an $x\in N$ such that $h_1(x)=y$. We define $h(y)$ as $h_2(x)$. To see that $h$ is well-defined, take $x'\in N$ such that $h_1(x')=y$. Then we have that the elements $x\to x', x'\to x, \sim x\to\sim x'$ and $\sim x'\to\sim x$ are all in $Ker(h_1)\subseteq Ker(h_2)$, so $h_2$ calculated in all of them is $1$, so $h_2(x)=h_2(x')$. It's clear from the construction that $h\circ h_1=h_2$, and $h$ is surjective. 
	
	It remains to  check that $h$ is a homomorphism. We prove it preserves the implication operation, the rest can be computed in a similar fashion. 
	
	Given $y, y'\in N_1$, let $x, x'$ and $z\in N$ be such that $h_1(x)=y, h_1(x')=y'$, and $h_1(z)=y \to y'$, respectively.   Then $h_1(x\to x')=h_1(x)\to h_1(x')=h_1(z)$. Since $Ker(h_1)\subseteq Ker(h_2)$, we have that $h(y\to y')=h_2(z)=h_2(x\to x')=h_2(x)\to h_2(x')=h(y)\to h(y')$.
	
	Finally, if $Ker(h_1)=Ker(h_2)$, by Lemma \ref{L62} we have that $N_1\cong N/Ker(h_1)= N/Ker(h_2)\cong N_2$.
\end{proof}

	The lattice structure of a Nelson algebra can be determined from the poset of its prime filters, using the Priestley representation. 	We briefly summarize here the part of this representation that we will use, following \cite{cignoli86class}. 	Since we will deal only with finite algebras when determining the free finitely generated five-valued Nelson algebras, we need not worry about the topology of their Priestley representation.
	
	The well-known Priestley duality is simplified in the finite case, and we have a duality between the category \lat\ of finite distributive lattices and their homomorphisms and the category \pos\ of finite posets and the order preserving functions between them. The functors that realize this duality are ${\sf X}:\lat\to\pos$ and ${\sf D}:\pos\to\lat$ defined as follows: if $L$ is a finite distributive lattice, let  $X(L)$ be the set of the  prime filters of $L$, and we set  $\fx(L)=(X(L), \subseteq)$. If $h:L\to L'$ is a lattice homomorphism, $\fx(h):\fx(L')\to\fx(L)$ is given by $\fx(h)(P)=h\mnu(P)$ for every $P\in \fx(L)$. Given a finite poset $S$, let $D(S)$ be the set  of all the decreasing subsets of $S$. Then $\fd(S)=(D(S), \cap,\cup,\emptyset, S)$ is a finite distributive lattice, and if $f:S\to S'$ is an order preserving function,   $\fd(f):\fd(S')\to \fd(S)$ given by $\fd(f)(U)=f\mnu(U)$ for all decreasing subsets $U$ of $S'$ is a lattice homomorphism.

	\begin{definition}
		We define \textit{Nelson spaces}, as triples $(X, \leq,\varphi)$ such that $X$ is a  finite set, $\leq$ is a partial order over $X$ and $\varphi$ is an involutive function from $X$ to $X$ that  is also a dual order isomorphism (this is, $\varphi(\varphi(x))=x$ and $x\le y$ if and only if $\varphi(y)\le\varphi(x)$) satisfying the conditions:
		\begin{itemize}
			\item For every $x\in X$, $x\leq\varphi(x)$ or $\varphi(x)\leq x$.
			\item (Interpolation property, A. Monteiro, {\rm \cite{monteiro63construction}}) If $x \leq \varphi(x)$, $x \leq \varphi(y)$,$y \leq \varphi(x)$, and $y \leq \varphi(y)$, then there exists $z\in X$ such that $x \leq z \leq\varphi(x)$ and $y \leq z \leq\varphi(y)$.
		\end{itemize}
	\end{definition}
	We consider a  category \Ns\  where the objects are Nelson spaces and the morphisms from a Nelson space $(X,\le,\varphi_X)$ to another, $(Y,\le,\varphi_Y)$ are the order preserving functions $f:X\to Y$ such that
	\begin{equation}\label{varphi}
	f\circ\varphi_X=\varphi_Y\circ f.
	\end{equation} 
	
	The functors \fx\ and \fd\ from above can be extended to yield a duality between the categories \textbf{N} and \Ns. If $A$ is a finite Nelson algebra, then $(X(A),\subseteq,\varphi_A)$ is a Nelson space, and if $(X,\le,\varphi)$ is a Nelson space, then $(D(X),X,\sim,\cap,\cup,\to)$ is a Nelson algebra if we define $\sim U=X\setminus\varphi(U)$ and $U\to V=X\setminus (U\cap \varphi(U)\cap(X\setminus V)]$ for all $U, V\in D(X)$, where for any subset $Y$ of $X$, $(Y]$ indicates the decreasing set generated by $Y$, that is, $(Y]=\set{x\in X| x\le y \text{ for some } y \in Y}$. Then we can define the functors $\fx_N:\textbf{N}\to \textbf{Ns}$ by $\fx(A)=(X(A),\subseteq,\varphi_A)$ for every finite Nelson algebra $A$, and $\fd_N(X)=$ \linebreak[3] $(D(X),X,\sim,\cap,\cup,\to)$ for every Nelson space $X$. These functors act on morphisms as \fx\ and \fd, respectively, and they give a dual adjunction between the categories \textbf{N} and \Ns.
	
	\begin{definition}
		Given a Nelson space $(X,\leq,\varphi)$, let $X^+=\set{x\in X|x\le\varphi(x)}$ and $X^-=\set{x\in X|\varphi(x)\le x}$. From the definition of Nelson space, we have that $X^+\cup X^-=X$, and by Lemmas \ref{L49} and \ref{L410}, if $N$ is a Nelson algebra, then $X(N)^+={\mathcal{I}}(N)$.
	\end{definition} 

\begin{lemma} \label{preservesPlus}
	If $f:X\to Y$ is a Nelson space morphism, and $x\in X^+$, then $f(x)\in Y^+$.
\end{lemma}
\begin{proof}
	If $x\le\varphi_X(x)$, then $f(x)\le f(\varphi_X(x))=\varphi_Y(f(x))$, so $f(x)\in Y^+$.
\end{proof}

\section{Free Algebras}\label{S4}
If $N$ is a Nelson algebra, and $X\subseteq N$, we denote with $S(X)$ the Nelson subalgebra  of $N$ generated by $X$.

The notion of free algebra is defined in the usual way, that is:
\begin{definition}
	Given a cardinal $\beta>0$ we say that  a five-valued Nelson algebra  $F$ has a set of  $\beta$ free generators, if
	\begin{enumerate}[\normalfont(L1)]
		\item $F$ contains a subset $G$ of cardinality $\beta$ such that $S(G) =F$,
		\item Every application $f:G \to N$, where $N$ is any five-valued Nelson algebra, can be extended to a unique Nelson algebra homomorphism $\bar{f}$, from $F$ to $N$.
	\end{enumerate}
	
	Under these conditions we say that $G$ is a set of free generators of $F$ and a five-valued Nelson algebra is said to be \emph{free} if it has a set of free generators. To show explicitly the cardinal $\beta$ we write $F=F({\beta})$.
\end{definition}

Since the notion of five-valued Nelson algebra is defined through identities, a result of G. Birkhoff \cite{birkhoff95lattice} ensures that there exists the free five-valued Nelson algebra  $F({\beta})$  with a free set of generators $G=\{g_i\}_{i\in I}$ of cardinality $\beta$, and it is unique up to isomorphisms. If $n$ is a natural number, $n\geq 1$, we are going to determine the five-valued Nelson algebra with a set of $n$  free generators $G=\{g_1,g_2,\ldots,g_n\}$, and we will denote it with $F(n)$.

 Since $F(n)$ is nontrivial, then by Theorem \ref{CIdeductivesystem}, ${\cal I}(F(n))\neq \emptyset$.
By Observation \ref{R32}, the only Nelson subalgebras of $C_5$, are $S_2=\{0,1\}\cong C_2$, $S_3=\{0,\cc,1\}\cong C_3 $, $S_4=\{0,\uc,\tc,1\}\cong C_4$, and $S_5=C_5$. 

If $f\in C_5^{G}$, that is, $f$ is a function from $G$ to $C_5$,  there exists an unique homomorphism $\bar{f}:F(n)\to C_5$ such that $\bar{f}(g)=f(g)$ for all $g\in G$, so by Lemma \ref{L61}, $S(\bar{f}(G))=\bar{f}(F(n))$ and since $S(\bar{f}(G))$ is a subalgebra of $C_5$, then $S(\bar{f}(G))\cong S_i$, for some $i$, $2\leq i\leq 5$. By Lemma \ref{LM1}, $Ker(\bar{f})\in {\cal I}(F(n))$. Writing $\alpha (f)=P_f= Ker (\bar{f})$, we have that \linebreak[4] $\alpha:C_5 ^G\to {\cal I}(F(n))$.

\begin{remark}\label{primeFiltersC5}
	The prime filters of $C_5$ are $[1), [\tc), [\cc)$, and  $[\uc)$. Of these, only $[1)$ and $ [\tc)$ are irreducible deductive systems, with $\varphi_{C_5}([1))=[\uc)$, and $\varphi_{C_5}([\tc))=[\cc)$.
\end{remark}

\begin{lemma}\label{determination}
If $N$ is a non-trivial five-valued Nelson algebra $h:N\to C_5$ a homomorphism, $D=Ker(h)$, and $M$ is the only maximal deductive system containing $D$, then we have: $D=h\mnu([1))$, $M=h\mnu([\tc))$, $\varphi(M)=h\mnu([\cc))$, and  $\varphi(D)=h\mnu([\uc))$.

As a consequence,
 $D=h\mnu([1))$,
 $M\setminus D=h\mnu(\set{\tc}))$, 
 $\varphi(M)\setminus M=h\mnu(\set{\cc})$,\linebreak[4]
 $\varphi(D)\setminus\varphi(M)=h\mnu(\set{\uc})$, and
 $N\setminus \varphi(D)=h\mnu(\set{0}))$.
\end{lemma}

\begin{proof}
	Notice in first place that we may have $D=M$, when $D$ is maximal, and also we may have $M=\varphi(M)$.
	
	It is direct from the definition of kernel that $D=h\mnu([1))$. Now let $M$ be the only maximal deductive containing $D$, which may coincide with $D$. By Lemma \ref{preservesPlus}, $h\mnu([\tc))$ is an irreducible deductive system, and it clearly contains $D$. To see that $h\mnu([\tc))$ is maximal as well, consider $x\notin h\mnu([\tc))$, and the deductive system $D(h\mnu([\tc)),x)=$ \linebreak[4] $\{y:x\to y\in h\mnu([\tc))\}$.  Since  $x\notin h\mnu([\tc))$, $h(x)\in\set{0,\uc,\cc}$. From the table of the implication for $C_5$, it follows that $h(x\to y)=$ $h(x)\to h(y)=1$ for every $y\in N$, so  $x\to y\in h\mnu([\tc))$ for all $y \in N$ and thus $D(h\mnu([\tc)),x)=N$. 

   By Corollary \ref{C42}, it follows that $M=h\mnu([\tc))$. Using equation (\ref{varphi}) and Observation \ref{primeFiltersC5}, we also get $\varphi(M)=h\mnu([\cc))$	and  $\varphi(D)=h\mnu([\uc))$.
   
  Since $h\mnu$ preserves intersections and complements, it follows that $h\mnu(\set{\tc})=$ \linebreak[4] $h\mnu([\tc)\setminus[1))=h\mnu([\tc))\setminus h\mnu([1))=M\setminus D$, and similarly we get $\varphi(M)\setminus M=h\mnu(\set{\cc})$ and $\varphi(D)\setminus\varphi(M)=h\mnu(\set{\uc})$.
\end{proof}

\begin{lemma}\label{L75} 
	The map $\alpha:C_5 ^G\to {\cal I}(F(n))$ is a bijection.
\end{lemma}

\begin{proof} 	We first prove that $\alpha$ is a surjection.  Let  $D \in {\cal I}(L(n))$, $L'=L(n)/D$, and $h_D$ the natural homomorphism from $L(n)$ onto $L'$, so $Ker(h_D)=D$. By Theorem \ref{T51a}, $L'\cong C_i,$ for some $i$, $2\leq i \leq 5$. Thus  $L'$ can be regarded as a subalgebra of $C_5$.
	
	Furthermore, $h_D$  determines a map $f$ from $G$ to $C_5,$ namely the restriction of $h_D$ to $G$:  $f(g)=  h_D(g),$ for every $g\in G$.
	Thus we have  $\bar{f}=h_D$ and therefore $\alpha(f)=Ker(\bar{f})=Ker(h_D)=D$.
	
	Now we show that $\alpha$ es injective as well. Let $f_1,f_2\in C_5^G$  be such that $\alpha(f_1)=Ker(\bar{f_1})=D=Ker(\bar{f_2})=\alpha(f_2)$. Then  we get that $f_1(g)=1$ iff $g\in D$ iff $f_2(g)=1$. Letting $M$ be as in Lemma \ref{determination}, we also get  $f_1(g)=\tc$ iff $g\in M\setminus D$ iff $f_2(g)=\tc$, and similarly for the other possible values of $f_1(g)$, we prove that it equals $f_2(g)$.
\end{proof}

\begin{lemma}\label{isomorfismo}
If $f:G\to C_5$, then 	$S(\bar{f}(G))\cong C_i$ iff $F(n)/P_f\cong C_i$.
\end{lemma}
\begin{proof}
	Since  $S(\bar{f}(G))=\bar{f}(F(n))$ and $\bar{f}(F(n))$ is isomorphic to $F(n)/P_f$, both algebras are isomorphic to the same subalgebra of $C_5$. 
\end{proof}

\begin{remark} \label{functions}
	Given $f:G\to C_5$, by the isomorphism from the previous lemma, we have that:
	\begin{enumerate}[\normalfont(I)]
		\item $S({f}(G))\cong C_2$ if and only if $f(g)\in\{0,1\}$ for all $g\in G$,
		\item $S({f}(G))\cong C_3$ if and only if $f(g)\in\{0,\cc,1\}$ for all $g\in G$ and there exists $g\in G$ such that $f(g)=\cc$,
		\item $S({f}(G))\cong C_4$ if and only if $f(g)\in\{0,\uc,\tc,1\}$ for all $g\in G$ and there exists $g\in G$ such that $f(g)=\uc$ or $f(g)=\tc$,
		\item $S({f}(G))\cong C_5$ if and only if there exists $g\in G$ such that $f(g)=\cc$ and there exists $g'\in G$ such that $f(g')=\uc$ or $f(g')=\tc$.
	\end{enumerate}
  We define $F_i=\set{f\in C_5^G:S(f(G))\cong C_i}$ for $i=2,3,4$, and $5$.
\end{remark}	

\textbf{Notation:} If $f\in C_5^G$, with $G=\set{g_1,\ldots,g_n}$, then we write: $f(g_i)=f_i$. $P_f={\alpha(f)=}Ker(\bar{f})$, where $\bar{f}$ is the only homomorphism from $F(n)$ to $C_5$ extending $f$.

\begin{remark}\label{remark_epis}
	Among the subalgebras of $C_5$ there are just two epimorphisms that are not isomorphisms. We can call them $\pi_e:C_4\to C_2$ and $\pi_o:C_5\to C_3$.
\[
\begin{array}{c||c|c|c|c}
	x       & 0 & \uc & \tc & 1 \\ \hline
	\pi_e(x) & 0 & 0  & 1 & 1 \\ 
\end{array}
\]

	\[
	\begin{array}{c||c|c|c|c|c}
		x       & 0 & \uc & \cc & \tc & 1 \\ \hline
		\pi_o(x) & 0 & 0 & \cc & 1 & 1 \\ 
	\end{array}
	\]
	
\end{remark}

The following theorem describes the order on the set $\mathcal{I}({F(n)})=X(F(n))^+$: 

\begin{theorem} Given  $f, h\in C_5^G$, we have that $P_h\subset P_f$ if and only if the following conditions are verified:
	\begin{enumerate}[\normalfont(1)]
		\item $f_i\in\set{0,\cc,1}$ for every $i$, $1\le i\le n$, 
		\item There exists at least an index $i$ such that $f_i\neq h_i$,
		\item if $f_i=0$ then $h_i=0$ or $h_i=\uc$,
		\item if $f_i=1$ then $h_i=1$ or $h_i=\tc$,
		\item $h_i=\cc$ if and only if $f_i=\cc$.
	\end{enumerate}
\end{theorem}

\begin{proof}
	Assume that the irreducible deductive systems satisfy $P_h\subset P_f$. We prove the conditions:
	\begin{enumerate}[\normalfont(1)]
		\item By Corollary \ref{C42a}, $P_f$ is maximal, and so by Theorem \ref{T51}, $F(n)/P_f$ is isomorphic to $C_2$ or $C_3$. Then, by Lemma \ref{isomorfismo}, and remark \ref{functions}, we have that $f(g)\in \set{0,\cc,1}$ for all $g\in G$. 
		\item If not then we would have $f=h$, so $P_f=P_h$, contradicting the hypothesis.
		\item If $f_i=0$, then $g_i\in\sim P_f$. Since $P_h\subset P_f$, by Lemma \ref{Lhom1} there is an epimorphism $e:F(n)/P_h\to F(n)/P_f$ such that $e\circ \bar{h}=\bar{f}$ . From $e\circ \bar{h}(g_i)=\bar{f}(g_i)=f_i$, and Remark \ref{remark_epis} it follows that $h_i$ must be $0$ or $\uc$.
		\item If $f_i=1$, by the same reasoning as in the previous item, $h_i$ must be $1$ or $\tc$. 
		\item As in the proof for part 3, we have an epimorphism  $e:F(n)/P_h\to F(n)/P_f$ such that $e\circ \bar{h}=\bar{f}$. Then $f_i=\bar{f}(g_i)=\cc=e(\bar{h}(g_i))=e(h_i)$ if and only if $h_i=\cc$. 
	\end{enumerate}
	
	In the other direction, assume $f$ and $h$ are such that conditions 1 to 5 hold. We want to prove that $P_h\subset P_f$. 
	
	By condition 1, for every $i$, $f\in\set{0,\cc,1}$.  It follows that $F(n)/P_f$ is isomorphic to $C_2$ or $C_3$. We also know by Theorem \ref{T51a}  that $F(n)/P_h$ is isomorphic to $C_k$ for $k\in\set{2,3,4,5}$. 
	
	\underline{Case 1}: If $F(n)/P_h\cong C_2$, then (A) $h_i\in\set{0,1}$ for all $i$. If for some $i$, $f_i=\cc$, then by condition 5, $h_i=\cc$, a contradiction. Then $f_1\in\set{0,1}$ for all $i$. If $f_i=0$, then by condition 3, $h_i\in\set{0,\uc}$, and by (A), we must have $h_i=0$. Similarly, if $f_i=1$, by condition 4 it follows that also $h_i=1$. We have proved then that $h=f$, contradicting condition 2. We conclude this case is not possible.
	
	\underline{Case 2}: If $F(n)/P_h\cong C_3$ then (B) $h_i\in\set{0,\cc,1}$ for all $i$, and there exists $j$ such that $h_j=\cc$. Now, if $h_i=\cc$ then by condition 5, $f_i=\cc$. If $h_i=0$, then $f_i$ cannot be $\cc$ also by condition 5.  If $f_i=1$, then by condition 4, $h_i=0\in\set{1,\tc}$, a contradiction. So we must have $f_i=0=h_i$. Similarly, if $h_i=1$ we can deduce using condition 3 that $f_i=1=h_i$, proving that $f=h$ and thus contradicting condition 2, so this case is not possible either.
	
	\underline{Case 3}: If $F(n)/P_h\cong C_4$ then $h_i\in\set{0,\uc,\tc,1}$ with at least one $h_i$ in $\set{\uc,\tc}$. By condition 5 this implies (C) $f_i\in\set{0,1}$ for all $i$, and therefore $F(n)/P_f\cong C_2$.
	\[
	\xymatrix{
		F(n)\ar[r]^{\bar{f}}\ar[d]_{\bar{h}} & F(n)/P_f\cong C_2\\
		F(n)/P_h\cong C_4\ar[ur]_{\pi_e}	& }
	\]
	We prove now that for all the generators $g_i\in G$ $\pi_e(\bar{h}(g_i))=\bar{f}(g_i)$, so $\pi_e(h_i)=f_i$.
	
	If $h_i=0$, we cannot have $f_i=1$ because in that case, by condition 4, we would have $h_i\in\set{1,\tc}$. Therefore, by (C), $f_i=0=\pi_e(h_i)$.
	
	If $h_i=\uc$, we cannot have $f_i=1$, because that would imply $h_i\in\set{1,\tc}$ as above. Then $f_i=0=\pi_e(\uc)=\pi_e(h_i)$.
	
	Similarly, if $h_i=\tc$, we cannot have $f_i=0$, because that would imply $h_i\in\set{0,\uc}$. Thus $f_i=1=\pi_e(\tc)=\pi_e(h_i)$.
	
	Finally, if $h_i=1$, we cannot have $f_i=0$, because that would imply $h_i\in\set{0,\uc}$. Thus $f_i=1=\pi_e(1)=\pi_e(h_i)$.
	
	Now we can see that if $x\in P_h=Ker(\bar{h})$, $\bar{h}(x)=1$ so $\pi_e(\bar{h}(x))=1=\bar{f}(x)$ so $x\in Ker(\bar{f})=P_f$.
	
	\underline{Case 4}:  If $F(n)/P_h\cong C_5$ then there is at least one $h_i$ in $\set{\uc,\tc}$, and one $h_j=\cc$. By condition 5 this implies $f_j=\cc$, and therefore $F(n)/P_f\cong C_3$. As in Case 3, we can prove that  $\pi_o(h_i)=f_i$ for all $i$ and therefore $\pi_o\circ \bar{h}=\bar{f}$.
	
	If $h_i=0$, we cannot have $f_i=1$ because in that case, by condition 4, we would have $h_i\in\set{1,\tc}$. By condition 5, $f_i$ cannot be $\cc$ either. Therefore, $f_i=0=\pi_o(0)=\pi_o(h_i)$.
	
	If $h_i=\uc$, we cannot have $f_i=1$, because that would imply $h_i\in\set{1,\tc}$ as above, nor $f_i=\cc$. Then $f_i=0=\pi_o(u)=\pi_o(h_i)$.
	
	If $h_i=\cc$, by condition 5 $f_i=\cc$. Then $f_i=\cc=\pi_o(\cc)=\pi_o(h_i)$.
	
	If $h_i=t$, we cannot have $f_i=0$, because that would imply $h_i\in\set{0,\uc}$ and also $f_i\neq \cc$. Thus $f_i=1=\pi_o(\tc)=\pi_o(h_i)$.
	
	Finally, if $h_i=1$, we cannot have $f_i=0$, because that would imply $h_i\in\set{0,\uc}$, and again $f_i\neq \cc$. Thus $f_i=1=\pi_o(1)=\pi_o(h_i)$.
	
	Now we can see as in Case 3 that $P_h\subset P_f$.
	
\end{proof}

\begin{example} If $n=2$, and we denote with pairs $(f_1,f_2)$ the functions in $C_5^2$, we can show the order of the irreductible deductive systems in $F(2)$:
	
	\setlength{\unitlength}{1mm}
	\begin{picture}(30,40)(-15,-20)                    
	\put(0,0){\circle{2}}
	\put(-10,0){\circle{2}}
	\put(0,10){\circle{2}}
	\put(10,0){\circle{2}}
	\put(-9,1){\line(1,1){8}}
	\put(9,1){\line(-1,1){8}}
	\put(0,1){\line(0,1){8}}
	\put(-5,-5){$(\uc,0)$}
	\put(-15,-5){$(0,\uc)$}
	\put(-5,14){$(0,0)$}
	\put(5,-5){$(\uc,\uc)$}
	\end{picture}
	\begin{picture}(30,40)(-15,-20)                    
	\put(0,0){\circle{2}}
	\put(-10,0){\circle{2}}
	\put(0,10){\circle{2}}
	\put(10,0){\circle{2}}
	\put(-9,1){\line(1,1){8}}
	\put(9,1){\line(-1,1){8}}
	\put(0,1){\line(0,1){8}}
	\put(-5,-5){$(\uc,1)$}
	\put(-15,-5){$(0,\tc)$}
	\put(-5,14){$(0,1)$}
	\put(5,-5){$(\uc,\tc)$}
	\end{picture}
	\begin{picture}(30,40)(-15,-20)                    
	\put(0,0){\circle{2}}
	\put(-10,0){\circle{2}}
	\put(0,10){\circle{2}}
	\put(10,0){\circle{2}}
	\put(-9,1){\line(1,1){8}}
	\put(9,1){\line(-1,1){8}}
	\put(0,1){\line(0,1){8}}
	\put(-5,-5){$(\tc,0)$}
	\put(-15,-5){$(1,\tc)$}
	\put(-5,14){$(1,0)$}
	\put(5,-5){$(\tc,\uc)$}
	\end{picture}
	\begin{picture}(30,40)(-15,-20)                    
	\put(0,0){\circle{2}}
	\put(-10,0){\circle{2}}
	\put(0,10){\circle{2}}
	\put(10,0){\circle{2}}
	\put(-9,1){\line(1,1){8}}
	\put(9,1){\line(-1,1){8}}
	\put(0,1){\line(0,1){8}}
	\put(-5,-5){$(\tc,1)$}
	\put(-15,-5){$(1,\tc)$}
	\put(-5,14){$(1,1)$}
	\put(5,-5){$(\tc,\tc)$}
	\end{picture}
	\begin{picture}(30,40)(-15,-20)                    
	\put(0,0){\circle{2}}
	\put(0,10){\circle{2}}
	\put(0,1){\line(0,1){8}}
	\put(-5,-5){$(\uc,\cc)$}
	\put(-5,14){$(0,\cc)$}
	\end{picture}
	\begin{picture}(30,40)(-15,-20)                    
	\put(0,0){\circle{2}}
	\put(0,10){\circle{2}}
	\put(0,1){\line(0,1){8}}
	\put(-5,-5){$(\cc,\uc)$}
	\put(-5,14){$(\cc,0)$}
	\end{picture}
	\begin{picture}(30,40)(-15,-20)                    
	\put(0,0){\circle{2}}
	\put(0,10){\circle{2}}
	\put(0,1){\line(0,1){8}}
	\put(-5,-5){$(\tc,\cc)$}
	\put(-5,14){$(1,\cc)$}
	\end{picture}
	\begin{picture}(30,40)(-15,-20)                    
	\put(0,0){\circle{2}}
	\put(0,10){\circle{2}}
	\put(0,1){\line(0,1){8}}
	\put(-5,-5){$(\cc,\tc)$}
	\put(-5,14){$(\cc,1)$}
	\end{picture}
	\begin{picture}(30,40)(-15,-20)                    
	\put(0,0){\circle{2}}
	\put(-5,-5){$(\cc,\cc)$}
	\end{picture}
\end{example}

Now we can make some observations on the order of the set ${\cal I}(F(n))$:
\begin{itemize}
	\item The functions in $\set{0,\cc,1}^G=F_2\cup F_3$ correspond to the maximal elements of ${\cal I}(F(n))$.
	\item Each of the functions in $\set{0,\cc,1}^G$ with $k$ occurrences of $\cc$ is above exactly $2^{n-k}-1$ irreducible deductive systems.
\end{itemize}

So far we have determined the order of the poset ${\cal I}(F(n))=X(F(n))^+$. In order to establish the structure of $X(F(n))$, notice that in general, in a Nelson space $X$, $X^-=\varphi(X^+)$, and recall that $\varphi$ is a dual order isomorphism, so we also know what is the order relation in $X(F(n))^-$. 

\begin{lemma} If $f\in\set{0,\cc,1}^G=F_2\cup F_3$, then $P_f$ is a maximal deductive system. Furthermore, if $f\in F_2$, $P_f=\varphi(P_f)$ while if $h\in F_3$,  $P_h\subset\varphi(P_h)$.
\end{lemma}

\begin{proof}
	Let $f\in F_2$. We already know that $P_f\subseteq\varphi(P_f)$. Now if $x\in\varphi(P_f)=\complement\sim P_f$, then $\sim x\notin P_f$, so $\bar{f}(\sim x)\neq 1$ and therefore $\bar{f}(x)\neq 0$. Since $\bar{f}(x)\in\set{0,1}$, it follows that $\bar{f}(x)= 1$ and $x\in P_f$.
	
	Now if $h\in F_3$, then by Lemma \ref{isomorfismo}, $S(\bar{h}(G))\cong C_3$. Consider \linebreak[5] $\fx(\bar{h}):X(C_3)\to X(F(n))$, so we have $\varphi_{F(n)}\fx(\bar{h})=\fx(\bar{h})\varphi_{C_3}$.  If we assume that $P_h=\varphi_{F(n)}(P_h)$, then we have that $\varphi_{F(n)}\fx(\bar{h})(\set{1})=\varphi_{F(n)}\bar{h}^{-1}(\set{1})=\varphi_{F(n)}(P_h)=P_h=\fx(\bar{h})\varphi_{C_3}(\set{1})=\fx(\bar{h})(\set{1,\cc})$. On the other hand, we know that for some $g\in G$, $h(g)=\cc$, so $g\in \fx(\bar{h})(\set{1,\cc})=\bar{h}^{-1}(\set{1,\cc})$ and $g\notin P_h$, so the sets cannot be equal.
\end{proof}

\begin{lemma}
	If $f,h\in C_5^G$, are such that $P_f$ and $P_h$ are different maximal elements in $\mathcal{I}(L(n))$, then they are not comparable in $X(L(n))$ with $\varphi(P_h)$ or $\varphi(P_f)$, respectively.
\end{lemma}
\begin{proof}
	Assume that $P_f\subseteq \varphi (P_h)$. Since $\varphi$ is a dual order isomorphism and an involution, it follows that $P_h\subseteq \varphi (P_f)$. We already know that $P_f\subseteq\varphi(P_f)$ and $P_h\subseteq\varphi(P_h)$. Then, by the interpolation property, there exists a prime filter $Q$ such that $P_f\subseteq Q\subseteq \varphi (P_f)$ and $P_h\subseteq Q\subseteq \varphi (P_h)$. Since either $Q$ or $\varphi(Q)\in X(L(n))^+$, and both fulfill the interpolation condition,  we can assume without loss of generality that $Q\in X(L(n))^+$.  So, from the maximality of $P_f$ and $P_h$ it follows that $P_f=P_h=Q$, a contradiction.
\end{proof}

\begin{example} \label{example_boolean} In $F(2)$, for the function $f=(0,0)$ we have that $\varphi(P_f)=P_f$ and therefore the connected component of $P_f$ in $X(F(2))$ is as depicted below:

	\setlength{\unitlength}{1mm}
	\begin{center}
\begin{picture}(30,60)(-15,-20)                    
	\put(0,0){\circle{2}}
	\put(-10,0){\circle{2}}
	\put(0,10){\circle{2}}
	\put(10,0){\circle{2}}
	
	\put(-10,20){\circle{2}}
	\put(0,20){\circle{2}}
	\put(10,20){\circle{2}}
	\put(-9,1){\line(1,1){8}}
	\put(1,11){\line(1,1){8}}
	\put(9,1){\line(-1,1){8}}
	\put(-1,11){\line(-1,1){8}}
	\put(0,1){\line(0,1){8}}
	\put(0,11){\line(0,1){8}}
	\put(-5,-5){$P_{(\uc,0)}$}
	\put(-15,-5){$P_{(0,\uc)}$}
	\put(5,10){$P_{(0,0)}=\varphi(P_{(0,0)})$}
	\put(5,-5){$P_{(\uc,\uc)}$}
	
	\put(-5,25){$\varphi(P_{(\uc,0)})$}
	\put(-21,25){$\varphi(P_{(0,\uc)})$}
	\put(11,25){$\varphi(P_{(\uc,\uc)})$}
\end{picture}
	\end{center}
\end{example}
\begin{example}\label{example_not_boolean}
For $h=(0,\cc)$, we have that $\varphi(P_h)\neq P_h$, so the corresponding connected component is:
\begin{center}
		\setlength{\unitlength}{1mm}
	\begin{picture}(30,40)(-15,-5)                  
	\put(0,0){\circle{2}}
	\put(0,10){\circle{2}}
    \put(0,20){\circle{2}}
    \put(0,30){\circle{2}}
	\put(0,1){\line(0,1){8}}
	\put(0,11){\line(0,1){8}}
	\put(0,21){\line(0,1){8}}
	\put(5,0){$P_{(\uc,\cc)}$}
	\put(5,10){$P_{(0,\cc)}$}
	\put(5,30){$\varphi( P_{(\uc,\cc)})$}
	\put(5,20){$\varphi(P_{(0,\cc)})$}
\end{picture}
\end{center}
\end{example}

Now we are ready to give a description of the free algebra $F(n)$. We start by describing the order structure of the poset $X(F(n))$. For each function in $\set{0,\cc,1}^G$, there is a maximal irreducible deductive system, and each one sits in  a different  connected component of $X(F(n))$. Thus, there are $3^n$ different connected components in $X(F(n))$. 

 When the function $f$ is in $F_2$, we have that the corresponding maximal deductive system $P_f$ is above exactly $2^{n}-1$ deductive systems $P_h$, where $h$ is any of the functions that can be obtained by replacing a nonempty subset of the coordinates $0$ or $1$ by $\uc$ or $\tc$, respectively. An illustration of this case was presented in Example \ref{example_boolean}. Each of these connected components of $X(F(n))^+=\mathcal{I}(F(n))$ has then $2^n$ elements. Since for such an $f$, $P_f=\varphi(P_f)$, it follows that the connected component in $X(F(n))$ has $2^n+2^n-1=2^{n+1}-1$ elements.
 
 If we consider now a function $f$ in $F_3$, then $P_f\neq\varphi(P_f)$. This is the case illustrated in Example \ref{example_not_boolean}. Let $k$ be the number of times that $\cc$ appears as a value of $f$. Then $P_f$ is above $2^{n-k}-1$ different irreducible deductive systems $P_h$ that correspond to the functions $h$ that can be obtained by replacing   a nonempty subset of the coordinates $0$ or $1$ in $f$ by $\uc$ or $\tc$, respectively. Each of these connected components of $X(F(n))^+=\mathcal{I}(F(n))$ has then $2^{n-k}$ elements. Since for such an $f$, $P_f\neq\varphi(P_f)$, it follows that the corresponding connected component in $X(F(n))$ has $2\cdot 2^{n-k}=2^{n-k+1}$ elements.

The algebra $F(n)$ can be obtained as the product of the Nelson algebras corresponding to each of the connected components of $X(F(n))$ that we have already determined. 

When the function $f$ is in $F_2$, the distributive lattice underlying the Nelson algebra we get can be thought of as the (underlying distributive lattice of) the boolean algebra with $2^n-1$ atoms with another copy of it on top. Each of these factors has $2\cdot 2^{2^n-1}=2^{2^n}$ elements. 

\begin{example}\label{ejebooleano}
	The diagram below shows the algebra obtained from the Nelson space from Example \ref{example_boolean}. The elements in black are the least elements of the prime filters. 
\setlength{\unitlength}{1mm}
\begin{center}
	\begin{picture}(40,70)(-20,0)
		\put(0,0){\circle{2}}
		\put(-10,10){\color{black}\circle*{2}}
		\put(0,10){\color{black}\circle*{2}}
		\put(10,10){\color{black}\circle*{2}}
		\put(-10,20){\circle{2}}
		\put(0,20){\circle{2}}
		\put(10,20){\circle{2}}
		\put(0,30){\circle{2}}
		\put(1,1){\line(1,1){8}}
		\put(1,11){\line(1,1){8}}
		\put(-9,11){\line(1,1){8}}
		\put(-9,21){\line(1,1){8}}
		\put(-1,1){\line(-1,1){8}}
		\put(-1,11){\line(-1,1){8}}
		\put(9,11){\line(-1,1){8}}
		\put(9,21){\line(-1,1){8}}
		\put(0,1){\line(0,1){8}}
		\put(-10,11){\line(0,1){8}}
		\put(10,11){\line(0,1){8}}
		\put(0,21){\line(0,1){8}}
		\put(0,31){\line(0,1){8}}
				\put(0,40){\color{black}\circle*{2}}
				\put(-10,50){\color{black}\circle*{2}}
				\put(0,50){\color{black}\circle*{2}}
				\put(10,50){\color{black}\circle*{2}}
				\put(-10,60){\circle{2}}
				\put(0,60){\circle{2}}
				\put(10,60){\circle{2}}
				\put(0,70){\circle{2}}
				\put(1,41){\line(1,1){8}}
				\put(1,51){\line(1,1){8}}
				\put(-9,51){\line(1,1){8}}
				\put(-9,61){\line(1,1){8}}
				\put(-1,41){\line(-1,1){8}}
				\put(-1,51){\line(-1,1){8}}
				\put(9,51){\line(-1,1){8}}
				\put(9,61){\line(-1,1){8}}
				\put(0,41){\line(0,1){8}}
				\put(-10,51){\line(0,1){8}}
				\put(10,51){\line(0,1){8}}
				\put(0,61){\line(0,1){8}}
			\end{picture}
		\end{center}
\end{example}		
		When the function $f$ in $F_3$, with $k\ge 1$ of the values equal to $\cc$,  the distributive lattice underlying the Nelson algebra we get can be thought of as the (underlying distributive lattice of) the boolean algebra with $2^{n-k}-1$ atoms with one element on top and then  another copy of the boolean algebra above. Each of these factors has $2\cdot 2^{2^{n-k}-1}+1=2^{2^{n-k}}+1$ elements. 

\begin{example}\label{factor_non_boolean}
	For the chain obtained in Example \ref{example_not_boolean}, we obtain the chain $C_5$ as the corresponding factor in the algebra $F(2)$.
\end{example}

With the structure of each factor already determined, we can count how many of these factors there are and the total number of elements of the free algebra.

There are $2^n$ functions in $\set{0,1}^G$, and the factor corresponding to each of these function has $2^{2^n}$ elements, so in this part of the free algebra there are $(2^{2^n})^{2^n}=2^{2^{2n}}$ elements.
 
 For the rest of the functions corresponding to maximal elements of $\mathcal{I}(F(n))$, notice that there are ${n\choose k}$ with $k$ occurrences of $\cc$, and therefore ${\binom{n}{k}2^{n-k}}$ factors of this kind. So for each value of $k, 1\le k\le n$, there are $(2^{2^{n-k}}+1)^{\binom{n}{k}2^{n-k}}$ elements in $F(n)$. In this way, we arrive to the formula:
	\[|F(n)|=2^{2^{2n}}\times\prod_{k=1}^{n}(2^{2^{n-k}}+1)^{\binom{n}{k}2^{n-k}}.\]
	
It is worth noting that the formula above had already appeared in the work of Diana Brignole, \cite{brignole65nelson}, indicating that she was indeed aware of the structure of the free algebras. However, to the best of our knowledge, she never published a proof or a detailed account of the underlying constructions. Her result, though stated without justification, aligns with the findings presented here and highlights the depth of her insight into the topic.

\bibliography{nelson} 
\bibliographystyle{alpha}

\appendix

\newpage
\setcounter{page}{1}

\section{Revista de la UMA, vol. 23 (1965), p. 46}
\label{S10}
\medskip

\noindent D. BRIGNOLE DE MARTIN (Universidad Nacional del Sur).\\
{\bf Algebras de Nelson pentavalentes.} (Por ausencia del autor se ley\'o el t\'{\i}tulo).
\medskip

Un \'algebra de Nelson se dice pentavalente si cumple la f\'ormula:\\             
$ ((a\rightarrow c)\rightarrow b)\rightarrow(((b \rightarrow a)\rightarrow b)\rightarrow b)=1.$ Los ejemplos m\'as sencillos, no triviales, de tales \'algebras, son una cadena de 5 elementos, y sus
sub\'algebras con 2, 3 y 4 elementos.
\medskip

TEOREMA FUNDAMENTAL: Toda \'algebra de Nelson pentavelente es sub\'algebra de un producto cartesiano de cadenas con 5 elementos.
\medskip

En la demostraci\'on de este teorema, desempe\~{n}a un papel importante el estudio de los sistemas deductivos, en particular de los sistemas deductivos irreductibles o primos.   

TEOREMA: Para que un \'algebra de Nelson sea pentavalente es necesario y suficiente que cada sistema deductivo irreductible sea m\'aximo, o est\'e contenido en un solo sistema deductivo propio.

Si a los axioma-esquema del c\'alculo proposicional constructivo con negaci\'on fuerte, se agrega el axioma-esquema
$$ ((a\rightarrow c)\rightarrow b)\rightarrow(((b \rightarrow a)\rightarrow b)\rightarrow b)=1.$$
el \'algebra de Lindenbaum correspondiente puede ser caracterizada como el \'algebra de Nelson pentavalente libre con tantos generadores libres como son  las variables de enunciado.

La cadena con 5 elementos algebrizada en su forma natural, es una matriz caracter\'{\i}stica para este c\'alculo.

Estos resultados son an\'alogos a los obtenidos por Luiz Monteiro, 1963, para el c\'alculo proposicional implicativo trivalente.

\setcounter{page}{1}

\newtheorem{defi}{\bf Definition}

\newtheorem{teorema}{\bf Theorem}

\newtheorem{corolario}{\bf Corollary}

\newenvironment{Dem.}{\noindent\bf Proof. \rm}{\hfill$ \square$}

\section{{On the 5-valued Nelson algebras}, preprint, Universidad Nacional del Sur, (1965).}
\label{S11}

\medskip

The purpose of this paper is to give a representation theorem for 5-valued Nelson algebras  and to express the number  of elements of the free algebra with $n$ generators.
\medskip

We will first recall the definitions of a Nelson algebra and after that we will prove some theorems  concerning the special class of the 5-valued algebras. (see \cite{{B2}, {B3},{B4}}, and \cite {B6} for the definitions and properties of a Nelson algebra.)

\medskip

{\bf I.- Definitions}\\[2mm]
We will use here the notion of deductive system, (i.e. a subset $D$ of $A$ such that i) $1\in D$ and ii) if $a,a\rightarrow b \in D$ then $b\in D$.), and those of maximal, irreducible, and completely irreducible deductive systems, and we will use properties concerning them, which are similar to those that can be proved for the same systems in a distributive lattice.

\begin{defi} A Nelson algebra  is a  system $<A, \wedge, \vee, \rightarrow,\sim, 1>$  where $A$ is a non empty set, $\wedge, \vee, \rightarrow$ are binary operations, $\sim$ is an unary operation, $1$ is a fixed element of $A$, and such that the following properties are verified:
	
	\begin{tabular}{llll}
		{\rm N1)}&$ x \vee 1=1,$ &  {\rm N2)}&$x\wedge (x \vee y)=x,$ \\
		&&&\\
		{\rm N3)}&$x \wedge (y \vee z)=(z \wedge x)\vee (y \wedge x),$ &  {\rm N4)}&$\sim \; \sim x=x,$ \\
		&&&\\
		{\rm N5)}&$\sim(x \wedge y)= \; \sim x \; \vee \sim y,$ & {\rm  N6)}&$x \; \wedge \sim x=(x \; \wedge \sim x)\wedge (y \; \vee \sim y),$ \\
		&&&\\
		{\rm N7)}&$x\rightarrow x=1,$&{\rm N8)}&$(\sim x \vee y)\wedge  (x\rightarrow y)=\sim x  \vee y,$ \\
		&&&\\
		{\rm N9)}&$x \wedge (x\rightarrow y)=x \wedge (\sim x \vee y),$& {\rm N10)}&$ x\rightarrow (y\wedge z)=(x \rightarrow y)\wedge (x \rightarrow z). $
	\end{tabular}
\end{defi}

\medskip

It follows from properties N1)-N5) that a Nelson algebra is a distributive lattice, with first and last elements, and with a Morgan negation.
\begin{defi}
	A Nelson algebra is said to be  5-valued if on it is verified the following formula:
	$$ ((a\rightarrow c)\rightarrow b)\rightarrow (((b \rightarrow a)\rightarrow b)\rightarrow b)=1.$$
\end{defi}

The identity above was suggested by the fact that it characterizes the 3-valued Heyting algebras (see \cite{B5}). Indeed, in most of this paper we will follow the ideas in
\cite{B5}.

\medskip

The simplest example of such an algebra is a chain of 5 elements in which the operations $\rightarrow,\sim$ are defined by the following tables, ($\wedge, \vee$ as usual)

$$
\begin{minipage}{4cm} 
	\beginpicture
	\setcoordinatesystem units <3mm,3mm>
	\setplotarea x from 0 to 20, y from 0 to 12
	\put {$ \bullet$} [c] at  8 0
	\put {$ 0$} [c] at  9 0
	\put {$ \bullet$} [c] at 8 3
	\put {$ a$} [c] at  9 3
	
	\put {$ \bullet$} [c] at  8 6
	\put {$ b$} [c] at  9 6
	\put {$ \bullet$} [c] at  8 9 
	\put {$ c$} [c] at  9 9
	\put {$ \bullet$} [c] at 8 12
	\put {$ 1$} [c] at  9 12
	\setlinear \plot 8 0  8 12 / 
	\endpicture
\end{minipage}
\begin{minipage}{3cm} 
	\beginpicture
	\setcoordinatesystem units <3mm,3mm>
	\setplotarea x from 0 to 10, y from 0 to 1
	\begin{tabular}{c|c}
		$x$&$\sim x$\\\hline
		$0$&$1$\\[2mm]
		$a$&$c$\\[2mm]
		$b$&$b$\\[2mm]    
		$c$&$a$\\[2mm]
		$1$&$0$
	\end{tabular}
	\endpicture
\end{minipage}
\begin{minipage}{4cm} 
	\beginpicture
	\begin{tabular}{c|ccccc}
		$\rightarrow$&$0$&$a$&$b$&$c$&$1$\\\hline
		$0$          &$1$&$1$&$1$&$1$&$1$\\[2mm]
		$a$&$1$&$1$&$1$&$1$&$1$\\[2mm]
		$b$&$1$&$1$&$1$&$1$&$1$\\[2mm]    
		$c$&$a$&$a$&$b$&$1$&$1$\\[2mm]
		$1$&$0$&$a$&$b$&$c$&$1$
	\end{tabular}
	\endpicture
\end{minipage}
$$

We will represent this particular 5N algebra (5-valued Nelson algebra) by $C_5$.

\begin{teorema}
	\label{T1} Let $C$ a completely irreducible deductive system (c.i.d.s.) of a 5N algebra  ${\cal A}$, and let $D$ be a proper deductive system which contains $C$ properly.Then $D$ is a maximal deductive system (m.d.s.).
\end{teorema}

\begin{Dem.} Let $C$ be a c.i.d.s. . Then $C$ is maximal between the d.s. which dont't contain some fixed element $b$.
	
	Let $D$ be a proper d.s., $C \subset D$. Therefore $b \in D$. Assume, by contradiction, that $D$  is not maximal. Then there exists a proper d.s. $U$ , $D \subset U.$
	
	Let $a \in U\setminus D$. Hence $a \notin C.$ Because $C$ is maximal relative to $b$, and $a\notin C$, we have $a \rightarrow b \in C.$
	
	From $a \in U$, $a \rightarrow ( b \rightarrow a)=1 \in U$, we get (i) $b \rightarrow a \in U.$ Besides, (ii) $b \rightarrow a \notin D.$ (If $b \rightarrow a \in D$, from $b\in D$ we get $a \in D$ which is a contradiction.)
	
	From (i), (ii), $b \rightarrow a \in U\setminus D$ and therefore $b \rightarrow a \notin C.$ Hence (iii) $(b \rightarrow a)\rightarrow b \in C.$  Then, $((b \rightarrow a)\rightarrow b)\rightarrow b \notin C.$ (Otherwise, from (iii) we get $b\in C$,  which is a contradiction.)
	
	Let $c$ be not in $U$. Then $a \rightarrow c \notin U$. (Otherwise, from  $a\in U$, we get $c \in U$.)
	
	Hence $a \rightarrow c \notin C$, and therefore (iv)  $(a \rightarrow c)\rightarrow b \in C.$ Since it is a 5N algebra, we get from (iv) $((b \rightarrow a)\rightarrow b)\rightarrow b \in C,$
	which is a contradiction.
\end{Dem.}
\begin{corolario}
	\label{C1}
	If a c.i.d.s. $C$ is not maximal there exists a unique proper d.s. which contains it. (And this one is maximal.)
\end{corolario}
\begin{teorema}
	\label{T2}  Let $A$ be an Nelson algebra with the following property: any c.i.d.s. either is maximal or there exists only one proper d.s. which contains it. Then $A$ is a 5N algebra.
\end{teorema}
\begin{proof}
	Assume by contradiction that there exists elements $a,b,c \in A$ such that:
	$${\rm (i)} \;\; ((a\rightarrow c)\rightarrow b)\rightarrow(((b \rightarrow a)\rightarrow b)\rightarrow b)\neq 1.$$
	Let $D_0$ be the d.s. generated by the elements  $(b \rightarrow a)\rightarrow b$ and $(a \rightarrow c)\rightarrow b$, i.e.
	$$ D_0=\{x\in A: ((a\rightarrow c)\rightarrow b))rightarrow(((b \rightarrow a)\rightarrow b)\rightarrow x)= 1\}.$$
	Because (i) $b\notin D_0$. Therefore, there exists a c.i.d.s. $C$ such that $b\notin C$  and $D_0\subseteq C$. $b\rightarrow a \notin C$. Therefore $a \notin C$.
	\medskip
	
	Let $D_1$ be the deductive system generated by $C$ and $b$, i.e. $D_1=\{x\in A : b \rightarrow x \in C\}$. Then $C\subseteq D_1$, and since $a\notin D_1$, $D_1$ is proper. In the same way we can define $D_2$ be the d.s. generated by $C$ and $a$, and prove that $C\subseteq D_2$ and that $D_2$ is proper.
	
	Therefore, we have two proper d.s. $D_1$ and $D_2$, both containing $C$ properly. And they are different, because $a\in D_2$, $a \notin D_1$. So, we get a contradiction. 
\end{proof}
\begin{corolario}
	\label{C2} A Nelson algebra is 5-valued iff any c.i.d.s. which is not maximal, is contained in only one proper d.s.
\end{corolario}
\begin{teorema}
	\label{T3}  Given an irreducible deductive system (i.d.s.) $C$ of a Nelson algebra, the family of all proper d.s.  which contains it, is a chain. 
\end{teorema}
\begin{proof} Let $K=\{D_i\}_{i\in I}$ the family of all proper d.s.  such that $C\subseteq D_i$. $C\in K$, so $K$ is not empty.
	
	Suppose by contradiction that $K$ is not a chain, i.e., there exists $i,j \in I$ such that: $D_i\not\subseteq D_j$, $D_j\not\subseteq D_i$.
	
	Therefore, there exists $a,b\in A$ such $a \in D_i$, $a\notin D_j$, $b \in D_j$, $b\notin D_i$. Hence $a,b \notin C$. Then $a\rightarrow b\in C$ and therefore $a\rightarrow b \in D_i$.
	From $a, a\rightarrow b\in D_i$, we get $b\in D_i$, which is a contradiction.
\end{proof}
\begin{teorema}
	\label{T4}   Any i.d.s. of a 5N algebra is a c.i.d.s.
\end{teorema}
\begin{proof} Let $C$ be an i.d.s.. Any proper d.s. can be given as the intersection of a family of c.i.d.s. . So, $C=\bigcap\limits_{i \in I} D_i$, for some $I$, $D_i$ c.i.d.s. for all $i \in I$.
	
	From Theorem \ref{T3}, $\{D_i\}_{i\in I}$ is a chain and from Theorem \ref{T1} it follows  that it has no  more than two elements, i.e., $C\subseteq D_1\subseteq D_2$. Therefore  $C=D_1$, i.e., $C$ is a c.i.d.s..
\end{proof}
\begin{corolario}
	\label{C3} A Nelson algebra is 5-valued iff any i.d.s. either is maximal or it is contained in a unique proper d.s..
\end{corolario}

{\bf II- Quotient algebra}
\begin{defi}
	Let $A$ be a 5N algebra de Nelson, and let  $C$  be a d.s  of $A$. We define  $a\equiv b$ (mod $C$), for $a,b \in A$ if the following conditions are verified:
	$$a\rightarrow b, b\rightarrow a, \sim a\rightarrow \sim b, \sim b\rightarrow \sim a\in C.$$
	It is easy to verify that $\equiv$ is a congruence relation on $A$.The corresponding quotient algebra, which is a 5N algebra, will be represented by $A/C$.
\end{defi}
\begin{teorema}
	\label{T5}   Let $C$ be an i.d.s. of a 5N algebra $A$. Then $A/C$ is either $C_5$ or one of its subalgebras.
\end{teorema}
\begin{proof} In this proof we will use the fact of the existence of an involution $\varphi$ defined on the family of all prime filters of  a Nelson algebra in the following way: $\varphi(P) =\complement \sim P$, (i.e. the set-theoretical complement of $\sim P=\{x:\sim x \in P\}$), and which have the following property: either $P\subseteq \varphi(P)$ or $\varphi(P)\subseteq P.$
	
	It easy verify that the prime filters $P$ which verify $P\subseteq \varphi(P)$ are deductive systems. (For more details about this involution see \cite{B1} and \cite{B6}).
	
	Because the previous theorems we know that given an i.d.s. $C$ there exists a unique i.d.s. $D$ such that $C\subseteq D$. Therefore, we can write $C\subseteq D\subseteq \varphi(D)\subseteq \varphi(C)$. We will consider the following fours possible cases:
	\begin{itemize}   
		\item[1)] $C\subset D\subset \varphi(D)\subset \varphi(C)$,
		\item[2)] $C\subset D= \varphi(D)\subset \varphi(C)$,
		\item[3)] $C= D\subset \varphi(D)= \varphi(C)$,
		\item[4)] $C= D=\varphi(D)= \varphi(C)$,
	\end{itemize}
	
	and we will prove that we get respectively 5, 4, 3 and 2 equivalence classes.
	
	\begin{itemize}
		\item[1)] $C\subset D\subset \varphi(D)\subset \varphi(C)$. Then, the equivalence classes modulo $C$ are $C$, $D\setminus C$, $\varphi(D)\setminus D$,  $\varphi(C)\setminus \varphi(D)$ and $A\setminus \varphi(C)$. In fact:
		\begin{itemize}
			\item[a)] Let $a,b \in C$. From $a\in C$, $a\rightarrow (b\rightarrow a)=1\in C$ it follows (i)  $b\rightarrow a \in C$. In the same way: (ii) $a\rightarrow b \in C$.
			From $a\in C$, $a\rightarrow (\sim a\rightarrow \sim b)=1\in C$ it follows (iii) $\sim a\rightarrow \sim b\in C$. And in the same way (iv) $\sim b\rightarrow \sim a\in C$.
			From (i)-(iv): $a \equiv b$ (mod C).
			\item[b)] Let $a,b \in D\setminus C$. It is verifiable that $C$ is maximal between the d.s. which don't contain $a$. From $b\notin C$ it follows (i) $b\rightarrow a\in C$. In the  same way (ii) 
			$a\rightarrow b \in C$. From $a,b\in D$ it follows $\sim a, \sim b \notin \varphi (D)$ and from $a,b\notin C$ it follows $\sim a,\sim b\in \varphi(C).$ Therefore 
			$\sim a,\sim b\in \varphi(C)\setminus \varphi(D)$ and besides $\sim a,\sim b \notin D$.
			Because $D$ is a maximal d.s. it follows $\sim a\rightarrow\sim b\in  D$, and $\sim b\rightarrow\sim a\in  D$. 
			
			Assume by contradiction $\sim a\rightarrow\sim b\notin  C$. Then, using the same argument that before, $C$ is maximal between those which don't contains $\sim a\rightarrow\sim b$. Then, because $\sim a \notin C$, we get $\sim a\rightarrow\sim b=\sim a \rightarrow (\sim a\rightarrow\sim b)\in  C$, which is a contradiction.
			Therefore (iii) $\sim a\rightarrow\sim b\in  C$, and in the same way  (iv) $\sim b\rightarrow\sim a\in C.$ From (i)-(iv):   $a \equiv b$ (mod C).
			\item[c)] Let $a,b \in \varphi(D)\setminus D$. $D$ is maximal relative to $a$ and $b$, and therefore $a\rightarrow b \in D$ and $b\rightarrow a \in D$. Assume by contradiction that $a\rightarrow b \notin C$. Using the same argument that before, we get that $C$ is maximal relative to $a\rightarrow b$  and then, because $a\notin C$, we have $a\rightarrow b=a\rightarrow(a\rightarrow b) \in C$, which is a contradiction.  So (i) $a\rightarrow b \in C$, and in the same way  (ii) $b\rightarrow a \in C.$ From $a,b \in \varphi(D)\setminus D$ it follows $\sim a,\sim b \in \varphi(D)\setminus D$, and therefore  (iii) $\sim a \rightarrow\sim b \in C$ and (iv) $\sim b \rightarrow\sim a \in C$. From (i)-(iv): $a \equiv b$ (mod C).
			\item[d)] Let $a,b \in \varphi(C)\setminus \varphi(D)$. Therefore $\sim a,  \sim b\in D\setminus C$. We proved in b) that this implies $\sim a \equiv \sim b$ (mod C) and therefore
			$a \equiv b$ (mod C).
			\item[e)] Let $a,b \in A\setminus \varphi(C)$. Therefore $\sim a,  \sim b\in  C$. We proved in a) that this  implies $\sim a \equiv \sim b$ (mod C) and therefore
			$a \equiv b$ (mod C).
		\end{itemize}
		From a)-e) it follows that there exists at most five equivalence classes. Now, we will prove that $C,D\setminus C, \varphi(D)\setminus D, \varphi(C)\setminus \varphi(D)$ and $A\setminus\varphi(C)$ are these classes.
		\begin{itemize}
			\item[a')] If $a \in C$ and $a \equiv b$ (mod C) then $b\in C$.\\
			From $a, a \rightarrow b \in C$ it follows $b \in C.$
			\item[b')] If $a \in D\setminus C$ and $a \equiv b$ (mod C) then $b\in D\setminus C$.\\
			From $a\in D$, $a \rightarrow b \in C\subset D$ it follows $b \in D.$ Assume by contradiction $b \in C$. Then, because  $b \rightarrow a \in C$, we get  $a \in C$. Therefore, $b\in D\setminus C$.
			\item[c')] If $a \in \varphi(D)\setminus D$ and $a \equiv b$ (mod C) then $b\in \varphi(D)\setminus D$.\\
			Assume by contradiction $b\in D$. Then because $b \rightarrow a \in C\subset D$ we get $a \in D$. Therefore (i) $b\notin D$. Assume now, by contradiction $b\notin \varphi(D).$ Then $\sim b \in D$, and from $\sim b \rightarrow \sim a \in C\subset D$ we get $\sim a \in D$, which is a contradiction because from $a\in \varphi (D)$  if follows $\sim a \notin D$. Therefore (ii) $b\in \varphi(D).$ From (i), (ii): $b\in \varphi(D)\setminus C$.
			\item[d')] If $a \in \varphi(C)\setminus \varphi(D)$ and $a \equiv b$ (mod C) then $b\in \varphi(C)\setminus \varphi(D)$.\\
			From the hypothesis it follows $\sim a \in D\setminus C$ and $\sim a \equiv \sim b$. Then by b') $\sim b \in D\setminus C $ and therefore $b \in \varphi(C)\setminus \varphi(D)$.
			\item[e')] If $a \in A\setminus \varphi(C)$ and $a \equiv b$ (mod C) then $b\in A\setminus \varphi(C)$.
			
		\end{itemize}
		The proofs on the second, third and fourth cases follow inmediately from the above discussion. And this ends the proofs of the theorem.
	\end{itemize}
\end{proof}

{\bf III.- Representation Theorem}
\begin{teorema} Any 5N algebra is a subalgebra of a direct product of chains with at most five elements.
\end{teorema}
\begin{Dem.}  Let $E =\{C_i\}_{i \in I}$ be a separable family of i.d.s. of $A$, i.e. $\bigcap\limits_{i\in I} C_i =\{1\}$. Assume that $ A$ is a 5N algebra.
	
	Given any  $C\in E$, we know that $ A/C$ is a chain with at most five elements.
	
	Write ${\cal L}_C=  A/C$, and let  ${\cal A}=\prod\limits_{C\in E} {\cal L}_C$, and let $c$ be the natural homomorphism of $ A$  onto $ A/C$.
	
	Our purpose is represent $ A$ as a subalgebra of ${\cal A}$. Given $f \in A$, let $\psi(f) = F\in {\cal A}$, where $F=\{F(C)\}_{C\in E}$, with $F(C)= c(f)\in {\cal L}_C$.
	
	\begin{itemize}
		\item [a)] $\psi$ is a homomorphism of  $ A$ into ${\cal A}$, i.e.
		\begin{itemize}
			\item [1)] $\psi(f\vee g)= \psi (f) \vee \psi(g)$
			\item [2)] $\psi(f\wedge g)= \psi (f) \wedge \psi(g)$
			\item [3)] $\psi(f\rightarrow g)= \psi (f) \rightarrow \psi(g)$
			\item [4)] $\psi(\sim f)= \sim \psi (f)$, for all $f,g \in  A.$
		\end{itemize}
		In fact let $h =f\vee g$, $F= \psi(f)$, $G= \psi(g)$, $H= \psi(h)$; the $H(C)= c(h) =c(f\vee g)= c(f) \vee c(g)= F(C) \vee G(c)$ for any $C\in E$. Therefore 1) $H=F\vee G$, i.e.
		$\psi(f\vee g)= \psi (f) \vee \psi(g)$. In the same way we can verify the equalities 2) - 4). 
		\item[b)] $\psi$ is one-to-one.
		Assume $\psi(f)=\psi(g)$, i.e. $\psi(f)\rightarrow \psi(g)= 	\psi(f \rightarrow g)=1$,\\
		$\psi(g)\rightarrow \psi(f)= 	\psi(g \rightarrow f)=1$,
		$\sim \psi(f)\rightarrow \sim \psi(g)= 	\psi(\sim f \rightarrow \sim g)=1$, and $\sim \psi(g)\rightarrow \sim \psi(f)= 	\psi(\sim g \rightarrow \sim f)=1$.
		
		Therefore $f \rightarrow g, g \rightarrow f, \sim f \rightarrow \sim g, \sim g \rightarrow \sim f \in C$, for any $C\in E.$ Hence
		$f \rightarrow g, g \rightarrow f, \sim f \rightarrow \sim g, \sim g \rightarrow \sim f \in \bigcap\limits_{C\in E} C =\{1\}$, i.e.\\
		$f \rightarrow g= g \rightarrow f= \sim f \rightarrow \sim g= \sim g \rightarrow \sim f=1$, i.e. $f=g$.
	\end{itemize}
	
	Let $A'$ the family of all elements of ${\cal A}$ of the form $\{\psi(f)\}_{f\in A}$. From a), b) it follows that $A'$ is a subalgebra of ${\cal A}$ isomorphic to $A$.
\end{Dem.}
\medskip

\noindent {\bf IV.- Free algebra with $n$ generators}\\[2mm]
We can define in a natural way the concept of Free 5N algebra with a given set of generators.
\medskip
We have studied such algebra with a finite number of generators, and determined that the number of its elements is the following $N$, being $n$ the number of generators:
$$N = 2^{2^{2n}} \prod\limits_{m=1}^{n} \left((2^{2^{n-m}} +1)^{{n \choose m} 2^{n-m}}\right). $$
\medskip

\noindent We construct the free 5N algebra with $1$ generator, and we verified that the number of its elements is $48$.\\[2mm]
The proofs of these results will be indicated in a next paper.

\hfill\includegraphics[width=0.5\textwidth]{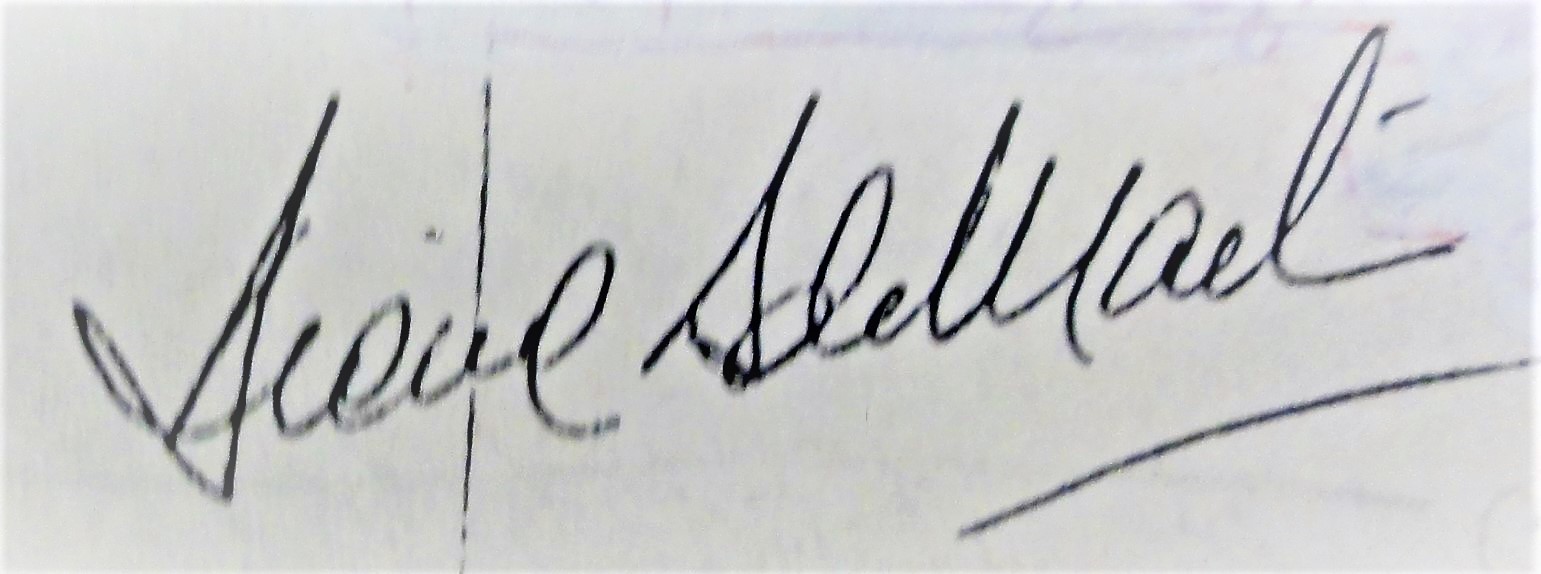}

\end{document}